\documentclass[preprint]{imsart}

\RequirePackage[OT1]{fontenc}
\RequirePackage{amsthm,amsmath}
\RequirePackage[numbers]{natbib}
\RequirePackage[colorlinks,citecolor=blue,urlcolor=blue]{hyperref}

\usepackage{natbib}
\usepackage{multirow}

\startlocaldefs
\numberwithin{equation}{section}
\theoremstyle{plain}
\newtheorem{thm}{Theorem}[section]
\newtheorem{cor}{Corollary}[section]
\newtheorem{lem}{Lemma}[section]
\newtheorem{rem}{Remark}[section]

\endlocaldefs

\usepackage{gensymb}
\usepackage{amssymb}

\usepackage{amssymb}
\usepackage{amsmath}
\usepackage{amsthm}

\begin{document}

\begin{frontmatter}
\title{Weak convergence theory for Poisson sampling designs}
\runtitle{Weak convergence theory for Poisson sampling}

\begin{aug}
\author{\fnms{Leo} \snm{Pasquazzi}\thanksref{t2}\ead[label=e1]{leo.pasquazzi@unimib.it}
\ead[label=u1,url]{http://www.fooBBBBBBBBBB.com}}

\thankstext{t2}{This work was supported by the grant 2016-ATE-0459 and the grant 2017-ATE-0402 from Universit\`{a} degli Studi di Milano-Bicocca. This work is an updated version of arXiv:1902.09169v1 [math.ST] and of arXiv:1902.09169v2 [math.ST]}
\runauthor{L. Pasquazzi}

\affiliation{Universit\`{a} degli Studi di Milano-Bicocca\thanksmark{m1}}

\address{Dipartimento di Statistica e Metodi Quantitativi\\
Universit\`{a} degli Studi di Milano-Bicocca\\
Edificio U7\\
Via Bicocca degli Arcimboldi, 8\\
20126 – Milano\\
\printead{e1}\\
\phantom{E-mail:\ }}

\end{aug}

\begin{abstract}
This work provides some general theorems about unconditional and conditional weak convergence of empirical processes in the case of Poisson sampling designs. The theorems presented in this work are stronger than previously published results. Their proofs are based on the symmetrization technique and on a contraction principle. 
\end{abstract}



\end{frontmatter}

\section{Introduction}

While there is now quite a deal of literature about functional central limit theorems for i.i.d. observations (see e.g. \cite{Pollard}, \cite{vdVW}, \cite{vdV}, \cite{Kosorok}), there are only a few papers about extensions which would be useful in the context of survey sampling. In fact, survey statisticians are often concerned with joint model and design-based inference and are therefore also interested in \textit{conditional} functional central limit theorems. Depending on the sampling design, survey statisticians can sometimes resort to results from the bootstrap theory which, however, seems to offer solutions only for the case where the sample inclusion indicator random variables are exchangeable (see e.g. \cite{Praestgaard_Wellner_1993} or \cite{vdVW}). This approach has been applied for example in \cite{Breslow_Wellner_2007} and \cite{Saegusa_Wellner_2013}. Extensions to other sampling designs have been investigated in \cite{Conti_2014}, \cite{Boistard_2017} and \cite{Bertail_2017}. \cite{Conti_2014} deals with \textit{high entropy designs}, i.e. sampling designs which can be approximated by rejective sampling, and provides sufficient conditions for weak convergence in $D[-\infty,\infty]$ (equipped with the Skorohod topology) of the Horvitz-Thompson and H{\'a}jek estimators of a finite population distribution function. \cite{Boistard_2017} considers the class of sampling designs for which the standardized Horvitz-Thompson estimator of the population mean of every uniformly bounded variable is asymptotically normal (almost surely conditional on the sequence of populations) and imposes restrictions on the first four mixed moments of the sample inclusion indicators in order to obtain some uniform central limit theorems for sequences of Horvitz-Thompson and H{\'a}jek estimators of a finite population distribution function. Finally, \cite{Bertail_2017} considers Poisson sampling and high entropy sampling designs as in \cite{Conti_2014}, and provides weak convergence theorems for Horvitz-Thompson empirical processes indexed by classes of functions $\mathcal{F}$ which satisfy the uniform entropy condition. 

The present paper is quite similar to \cite{Bertail_2017} but focuses only on Poisson sampling designs. However, the results presented in this paper are more general than those given in \cite{Bertail_2017}. In fact, the present paper provides functional central limit theorems which can be applied to a much wider family of function classes $\mathcal{F}$ than the one considered in  \cite{Bertail_2017} (the uniform entropy condition will not be required). Moreover, this paper considers also the case where the first order sample inclusion probabilities can be arbitrarily close to zero which has not been treated in \cite{Bertail_2017} and provides extensions for the H{\'a}jek empirical process.

This article is organized as follows. Section \ref{Section_notation_definitions} introduces the notation and the probability space within which the weak convergence theorems of this paper will be derived. The probabilistic framework that will be introduced in this section is completely general and can be used to derive weak convergence theorems for other sampling designs as well. The main results of this paper will then be derived in Section \ref{empirical_process_theory_POISSON}. The first battery of results deals with the case where there is a positive lower bound on the sample inclusion probabilities (which is the case considered in \cite{Bertail_2017}). Then an analogous set of results will be obtained for the case where the sample inclusion probabilities are proportional to some size variable which can take on values arbitrarily close to zero (this case has not been treated in \cite{Bertail_2017}). Extensions for H{\'a}jek empirical processes will then be derived in Section \ref{Hajek_empirical_process_theory_POISSON}. Finally, Section \ref{Simulation_results} concludes this paper with some simulation results.

\section{Notation and Definitions}\label{Section_notation_definitions}

Let $Y_{1}$, $Y_{2}$, \dots, $Y_{N}$ denote the values taken on by a study variable $Y$ on the $N$ units of a finite population, and let $X_{1}$, $X_{2}$, \dots, $X_{N}$ denote corresponding values of an auxiliary variable $X$. In this paper it will be assumed that the $N$ ordered pairs $(Y_{i}, X_{i})$ corresponding to a given finite population of interest are the first $N$ realizations of an infinite sequence of i.i.d. random variables which take on values in the cartesian product of two separable Banach spaces. The latter will be denoted by $\mathcal{Y}\times\mathcal{X}$. Moreover, as usual in finite population sampling theory, it will be assumed that the values taken on by the auxiliary variable $X$ are known in advance for all the $N$ population units, while those corresponding to the study variable $Y$ are only known for the population units that have been selected into a random sample. The corresponding vector of sample inclusion indicator functions will be denoted by $\mathbf{S}_{N}:=(S_{1,N}, S_{2,N}, \dots, S_{N,N})$, and it will be assumed that the vectors $\mathbf{S}_{N}$ and $\mathbf{Y}_{N}:=(Y_{1}, Y_{2}, \dots, Y_{N})$ are \textit{conditionally} independent \textit{given} $\mathbf{X}_{N}:=(X_{1}, X_{2}, \dots, X_{N})$. With reference to the sample design, probability and expectation will be denoted by $P_{d}$ e $E_{d}$, respectively. With this notation, the vector of first order sample inclusion probabilities will be given by
\begin{equation*}
\begin{split}
\underline{\mathbf{\pi}}_{N}&:=(\pi_{1,N}, \pi_{2,N}, \dots, \pi_{N,N})\\
&:=(E_{d}S_{1,N}, E_{d}S_{2,N}, \dots, E_{d}S_{N,N})\\
&=(P_{d}\{S_{1,N}=1\}, P_{d}\{S_{2,N}=1\},\dots, P_{d}\{S_{N,N}=1\}),
\end{split}
\end{equation*}
and from the conditional independence assumption it follows that $\underline{\mathbf{\pi}}_{N}$ must be a deterministic function of $\mathbf{X}_{N}$.

\smallskip

Now, with reference to the Banach space $\mathcal{Y}$ consider the random empirical measures given by
\begin{equation*}\label{HT_empirical_process}
\mathbb{G}_{N}':=\frac{1}{\sqrt{N}}\sum_{i=1}^{N}\left(\frac{S_{i,N}}{\pi_{i,N}}-1\right)\delta_{Y_{i}}.
\end{equation*}
For $f:\mathcal{Y}\mapsto\mathbb{R}$, the integral of $f$ with respect to $\mathbb{G}_{N}'$ can be written as
\[\mathbb{G}_{N}'f:=\frac{1}{\sqrt{N}}\sum_{i=1}^{N}\left(\frac{S_{i,N}}{\pi_{i,N}}-1\right)f(Y_{i})\]
so that, for any given class $\mathcal{F}$ of functions $f:\mathcal{Y}\mapsto\mathbb{R}$, the random empirical measure $\mathbb{G}_{N}'$, as a real-valued function of $f\in\mathcal{F}$, can be interpreted as a stochastic process indexed by the set $\mathcal{F}$. For obvious reasons $\mathbb{G}_{N}'$ will be called Horvitz-Thompson empirical process (henceforth HTEP). Depending on the values taken on by the study variable $Y$ and on the class of functions $\mathcal{F}$, a sample path of $\mathbb{G}_{N}'$ could be either bounded or not. In the former case it will be an element of $l^{\infty}(\mathcal{F})$, the space of all bounded and real-valued functions with domain given by the class of functions $\mathcal{F}$. In what follows $l^{\infty}(\mathcal{F})$ will be considered as a metric space with distance function induced by the norm $\lVert z\rVert_{\mathcal{F}}:=\sup_{f\in\mathcal{F}}|z(f)|$.

\smallskip

As already mentioned in the introduction, the present paper investigates conditions under which
\[\mathbb{G}_{N}'\rightsquigarrow \mathbb{G}'\text{ in }l^{\infty}(\mathcal{F}),\]
where $\mathbb{G}'$ is a Borel measurable and tight (in $l^{\infty}(\mathcal{F})$) Gaussian process. Both \textit{unconditional} and \textit{conditional} (on the realized values of $X$ and $Y$) weak convergence will be considered. Recall that unconditional weak convergence is defined as
\[E^{*}h(\mathbb{G}_{N}')\rightarrow Eh(\mathbb{G}')\quad\text{ for all }h\in C_{b}(l^{\infty}(\mathcal{F})),\]
where $C_{b}(l^{\infty}(\mathcal{F}))$ is the class of all real-valued and bounded functions on $l^{\infty}(\mathcal{F})$. For Borel measurable (in $l^{\infty}(\mathcal{F}$)) processes $\mathbb{G}'$ whose realizations lie in a separable subset of $l^{\infty}(\mathcal{F})$ almost surely this is equivalent to
\[\sup_{h\in BL_{1}(l^{\infty}(\mathcal{F}))}\left|E^{*}h(\mathbb{G}_{N}')-Eh(\mathbb{G}')\right|\rightarrow 0,\]
where $BL_{1}(l^{\infty}(\mathcal{F}))$ is the set of all functions $h:l^{\infty}(\mathcal{F})\mapsto[0,1]$ such that $|h(z_{1}-h(z_{2})|\leq \lVert z_{1}-z_{2}\rVert_{\mathcal{F}}$ for every $z_{1},z_{2}\in l^{\infty}(\mathcal{F})$ (see Chapter 1.12 in \cite{vdVW}). Based on this observation, \cite{vdVW} provides two definitions of conditional weak convergence: \textit{conditional weak convergence in outer probability} (henceforth opCWC), which in the context of this paper translates to the condition
\[\sup_{h\in BL_{1}(l^{\infty}(\mathcal{F}))}\left|E_{d}h(\mathbb{G}_{N}')-Eh(\mathbb{G}')\right|\overset{P^{*}}{\longrightarrow} 0\]
(see page 181 in \cite{vdVW}), and \textit{outer almost sure conditional weak convergence} (henceforth oasCWC), which in the context of this paper translates to the condition
\[\sup_{h\in BL_{1}(l^{\infty}(\mathcal{F}))}\left|E_{d}h(\mathbb{G}_{N}')-Eh(\mathbb{G}')\right|\overset{as*}{\rightarrow} 0.\]
As expected, oasCWC implies opCWC (see Lemma 1.9.2 on page 53 in \citep{vdVW}). However, it seems that in absence of asymptotic measurabilty of the sequence $\{\mathbb{G}_{N}'\}_{N=1}^{\infty}$ oasCWC is not strong enough to imply \textit{unconditional} weak convergence (cfr. Theorem 2.9.6 on page 182 in \cite{vdVW} and the comments thereafter).

\smallskip

Since the very definition of weak convergence relies on the concept of outer expectation, some assumptions about the underlying probability space will be necessary for what follows. Throughout this paper it will always be assumed that the latter is a product space of the form
\begin{equation}\label{probability_space}
\prod_{i=1}^{\infty} (\Omega_{y,x}, \mathcal{A}_{y,x}, P_{y,x})\times (\Omega_{d},\mathcal{A}_{d},P_{d})\times \prod_{i=1}^{\infty} (\Omega_{\varepsilon}, \mathcal{A}_{\varepsilon}, P_{\varepsilon})
\end{equation}
and that the elements of the random sequence $\{(Y_{i},X_{i})\}_{i=1}^{\infty}$ are the coordinate projections on the first infinte coordinates of the sample points $\omega\in\Omega_{y,x}^{\infty}\times\Omega_{d}\times\Omega_{\varepsilon}^{\infty}$. On the other hand, the independent Rademacher random variables $\varepsilon_{1}$, $\varepsilon_{2}$, \dots, which will be needed for symmetrization, are defined as the coordinate projections on the last infinite coordinates of the sample points $\omega\in\Omega_{y,x}^{\infty}\times\Omega_{d}\times\Omega_{\varepsilon}^{\infty}$. The joint expectation with respect to all the Rademacher random variables with the first $\infty+1$ coordinates of the sample points kept fixed will be denoted by $E_{\varepsilon}$. Finally, the sample inclusion indicators $S_{i,N}$ are allowed to depend on the first $\infty+1$ coordinates of the sample points $\omega\in\Omega_{y,x}^{\infty}\times\Omega_{d}\times\Omega_{\varepsilon}^{\infty}$ only. As suggested by the notation, it will be assumed that for each value of $N$ the corresponding sample inclusion indicator functions $S_{1,N}$, $S_{2,N}$, \dots, $S_{N,N}$ are the elements of one row of a triangular array of random variables. This assumption is needed to make sure that for each value of $N$ the sample design can be adapted according to all the $N$ (known) values taken on by the auxiliary variable $X$ as the population size increases. To make sure that the conditional independence assumption holds, it will be assumed that for each value of $N$ the corresponding vector $\mathbf{S}_{N}$ is defined as a function of the random vector $\mathbf{X}_{N}$ and of random variables $D_{1}$, $D_{2}$, \dots which are functions  of the central coordinate of the sample points $\omega\in\Omega_{y,x}^{\infty}\times\Omega_{d}\times\Omega_{\varepsilon}^{\infty}$ only, i.e. of the coordinate that takes on values in the set $\Omega_{d}$ (instead of a random sequence $\{D_{i}\}_{i=1}^{\infty}$ one could also consider a stochastic process $\{D_{t}: t\in T\}$ with an arbitrary index set $T$ but this will not be of interest in the present paper). For example, in the case of a Poisson sampling design with a given vector of first order sample inclusion probabilities $\mathbf{\underline{\pi}}_{N}$ we could define $\{D_{i}\}_{i=1}^{\infty}$ as a sequence of i.i.d. uniform-$[0,1]$ random variables and define for each value of $N$ the corresponding row of sample inclusion indicators by
\[S_{i,N}:=\begin{cases}
1&\text{if }D_{i}\leq \pi_{i,N}\\
0 & \text{ otherwise}
\end{cases}
\quad\quad i=1,2,\dots, N\]
Of course, the above probability space does not only work for Poisson sampling designs, but it can accommodate any non-informative sampling design. To prove this assertion it will be shown that \textit{for any non-informative sampling design the vector of sample inclusion indicators $\mathbf{S}_{N}$ can be defined as a function of $\mathbf{X}_{N}$ and of a single uniform-$[0,1]$ random variable $D$ that depends on the central coordinate of the sample points $\omega\in\Omega_{y,x}^{\infty}\times\Omega_{d}\times\Omega_{\varepsilon}^{\infty}$ only}. To this aim let
\begin{equation}\label{sample_selection_probability}
\mathfrak{p}_{N}(\mathbf{s}_{N}):=\mathfrak{p}_{N}(\mathbf{s}_{N}; \mathbf{X}_{N})
\end{equation}
denote the probability to select a given sample $\mathbf{s}_{N}\in\{0,1\}^{N}$. Note that the definition of the function $\mathfrak{p}_{N}$ specifies a desired sampling design. Since the values taken on by the auxiliary variable $X$ are assumed to be already known before the sample is drawn, the sample selection probabilities $\mathfrak{p}_{N}(\mathbf{s}_{N})$ are allowed to depend on $\mathbf{X}_{N}$. Now, let $\mathbf{s}_{N}^{(1)}$, $\mathbf{s}_{N}^{(2)}$, \dots, $\mathbf{s}_{N}^{(2^{N})}$ denote the $2^{N}$ elements of $\{0,1\}^{N}$ arranged in some fixed order (for example, according to the order determined by the binary expansion corresponding to the finite sequence of zeros and ones in $\mathbf{s}_{N}$), and put $\mathfrak{p}_{N}^{(i)}:=\mathfrak{p}_{N}(\mathbf{s}_{N}^{(i)})$, $i=1,2,\dots, 2^{N}$. Then, define the vector of sample inclusion indicators $\mathbf{S}_{N}$ by
\[\mathbf{S}_{N}:=\begin{cases}
\mathbf{s}_{N}^{(1)}\quad\text{ if }D\leq \mathfrak{p}_{N}^{(1)}\\
\mathbf{s}_{N}^{(i)}\quad\text{ if }\sum_{j=1}^{i-1}\mathfrak{p}_{N}^{(j)}<D\leq\sum_{j=1}^{i}\mathfrak{p}_{N}^{(j)}\text{ for } i=2,3,\dots, 2^{N},
\end{cases}\]
and note that for every $\mathbf{s}_{N}\in\{0,1\}^{N}$ this vector satisfies $P_{d}\{\mathbf{S}_{N}=\mathbf{s}_{N}\}=\mathfrak{p}_{N}(\mathbf{s}_{N})$ as desired. This concludes the proof of the above assertion written in italics. 

Next, observe that in the above construction the sample selection probabilities $P_{d}\{\mathbf{S}_{N}=\mathbf{s}_{N}\}$ are functions of $\mathbf{X}_{N}$. If for a given $\mathbf{s}_{N}\in\{0,1\}^{N}$ the corresponding sample selection probability $P_{d}\{\mathbf{S}_{N}=\mathbf{s}_{N}\}$ is a measurable function of $\mathbf{X}_{N}\in\mathcal{X}^{N}$ (this depends on the sampling design), then, with reference to the probability space of this paper, $P_{d}\{\mathbf{S}_{N}=\mathbf{s}_{N}\}$ can be interpreted as a conditional probability in the proper sense. Otherwise, $P_{d}\{\mathbf{S}_{N}=\mathbf{s}_{N}\}$ will just be a non measurable (random) function of $\mathbf{X}_{N}$. More generally, the expectation with respect to the uniform random variable $D$, with $\mathbf{Y}_{N=\infty}$, $\mathbf{X}_{N=\infty}$ and $\varepsilon_{1}$, $\varepsilon_{2}$, \dots kept fixed, which can be interpreted as design expectation and will therefore be denoted by $E_{d}$, can be applied to \textit{any} function $g$ of $\mathbf{S}_{N}$, $\mathbf{Y}_{N}$, $\mathbf{X}_{N}$ and $\varepsilon_{1}$, $\varepsilon_{2}$, \dots, $\varepsilon_{N}$. In fact, the expectation
\[E_{d}g(\mathbf{S}_{N}, \mathbf{Y}_{N},\mathbf{X}_{N},\varepsilon_{1}, \varepsilon_{2},\dots, \varepsilon_{N})\]
is given by 
\[\sum_{\mathbf{s}_{N}\in\{0,1\}^{N}}g(\mathbf{s}_{N}, \mathbf{Y}_{N},\mathbf{X}_{N},\varepsilon_{1}, \varepsilon_{2},\dots, \varepsilon_{N})P_{d}\{\mathbf{S}_{N}=\mathbf{s}_{N}\},\]
and $E_{d}g(\mathbf{S}_{N}, \mathbf{Y}_{N},\mathbf{X}_{N},\varepsilon_{1}, \varepsilon_{2},\dots, \varepsilon_{N})$ is thus a function of $\mathbf{Y}_{N}$, $\mathbf{X}_{N}$ and $\varepsilon_{1}$, $\varepsilon_{2}$, \dots, $\varepsilon_{N}$. If for every fixed $\mathbf{s}_{N}\in\{0,1\}^{N}$ the corresponding function $g(\mathbf{s}_{N},\cdot)$ is a measurable function of $\mathbf{Y}_{N}$, $\mathbf{X}_{N}$ and $\varepsilon_{1}$, $\varepsilon_{2}$, \dots, $\varepsilon_{N}$ and the function $P_{d}\{\mathbf{S}_{N}=\mathbf{s}_{N}\}$ is a measurable function of $\mathbf{X}_{N}$, then, with respect  to the probability space of this paper, $E_{d}g(\mathbf{S}_{N}, \mathbf{Y}_{N},\mathbf{X}_{N},\varepsilon_{1}, \varepsilon_{2},\dots, \varepsilon_{N})$ can be interpreted as a conditional expectation in the proper sense (and in this case it will obviously be a measurable function of $\mathbf{Y}_{N}$, $\mathbf{X}_{N}$ and $\varepsilon_{1}$, $\varepsilon_{2}$, \dots, $\varepsilon_{N}$), while otherwise it could either be a measurable or a non measurable function of $\mathbf{Y}_{N}$, $\mathbf{X}_{N}$ and $\varepsilon_{1}$, $\varepsilon_{2}$, \dots, $\varepsilon_{N}$.

Throughout this paper it will be assumed that all the vectors of sample inclusion indicators $\mathbf{S}_{N}$ are defined as described in the above construction (the one which involves a single uniform-$[0,1]$ random variable $D$). Of course, in this way the random vectors $\mathbf{S}_{N}$ will be dependent for different values of $N$, but for the purposes of this paper this dependence structure is irrelevant. Moreover, in what follows only \textit{measurable sample designs} will be considered, i.e. sample designs such that for every fixed $\mathbf{s}_{N}\in\{0,1\}^{N}$ the corresponding sample selection probability in (\ref{sample_selection_probability}) is a measurable function of $\mathbf{X}_{N}$. Note that this is a very mild restriction that should be satisfied in virtually every practical setting. However, it entails three important consequences which will be relevant for the proofs presented in this paper. They are: (i) the vectors of sample inclusion indicators $\mathbf{S}_{N}$ are measurable functions of $\mathbf{X}_{N}$ and of the uniform-$[0,1]$ random variable $D$, (ii) for every $\mathbf{s}_{N}\in\{0,1\}^{N}$ the corresponding probability $P_{d}\{\mathbf{S}_{N}=\mathbf{s}_{N}\}$ is a conditional probability in the proper sense, and (iii) for $g$ a measurable function of $\mathbf{S}_{N}$, $\mathbf{Y}_{N}$, $\mathbf{X}_{N}$ and $\varepsilon_{1}$, $\varepsilon_{2}$, \dots, $\varepsilon_{N}$ the corresponding expectation $E_{d}g(\mathbf{S}_{N}, \mathbf{Y}_{N}, \mathbf{X}_{N}, \varepsilon_{1}, \varepsilon_{2}, \dots, \varepsilon_{N})$ is a conditional expectation in the proper sense. 

\smallskip

Finally, for the proofs presented in this paper it will be convenient to introduce a special kind of measurable cover function that could differ from the traditional one for functions that depend on the sample inclusion indicators. Recall that, according to the traditional definition (see for example Section 1.2 in \cite{vdVW}), a measurable cover of a real-valued function $T$ defined on some generic probability space $(\Omega,\mathcal{A},P)$ is a measurable function $T^{*}:\Omega\mapsto\mathbb{R}\cup\{-\infty,\infty\}$ such that (i) $T^{*}\geq T$, and (ii) $T^{*}\leq U$ $P$-almost surely, for every measurable $U:\Omega\mapsto\mathbb{R}\cup\{-\infty,\infty\}$ with $U\geq T$ $P$-almost surely. The existence of measurable cover functions follows e.g. from Lemma 1.2.1 on page 6 in \cite{vdVW}. Note that the very definition of measurable cover implies that it is unique only up to $P$-null sets. Moreover, since the definition of measurable cover function depends on the underlying probability measure $P$, it should actually be called measurable $P$-cover function. Now, consider the probability space of this paper and a function $T:\Omega_{y,x}^{\infty}\times\Omega_{d}\times\Omega_{\varepsilon}^{\infty}\mapsto\mathbb{R}$ that depends on $\omega\in\Omega_{y,x}^{\infty}\times\Omega_{d}\times\Omega_{\varepsilon}^{\infty}$ only through the vector of random elements 
\[\left(\mathbf{Y}_{N},\mathbf{X}_{N}, \mathbf{S}_{N}, \varepsilon_{1},\varepsilon_{2},\dots,\varepsilon_{N}\right).\]
Let 
\[\phi:\Omega_{y,x}^{\infty}\times\Omega_{d}\times\Omega_{\varepsilon}^{\infty}\mapsto\Omega_{y,x}^{N}\times\{0,1\}^{N}\times\Omega_{\varepsilon}^{N}\]
be the map that transforms the sample points $\omega\in\Omega_{y,x}^{\infty}\times\Omega_{d}\times\Omega_{\varepsilon}^{\infty}$ into the above vector that is relevant for the computation of $T$, and note that $T$ can always be written as $T=h\circ\phi$ for some $h:\Omega_{y,x}^{N}\times\{0,1\}^{N}\times\Omega_{\varepsilon}^{N}\mapsto\mathbb{R}$. Then, with $h^{*}$ a measurable $(P_{y,x}^{\infty}\times P_{d}\times P_{\varepsilon}^{\infty})\circ\phi^{-1}$-cover of $h$, define a majorant of $T$ by $T^{**}:=h^{*}\circ\phi$. Then, (i) $T^{**}:\Omega_{y,x}^{\infty}\times\Omega_{d}\times\Omega_{\varepsilon}^{\infty}\mapsto\mathbb{R}\cup\{\infty,-\infty\}$ is measurable because $\phi$ and $h^{*}$ are both measurable, (ii) $T^{**}\geq T^{*}$ almost surely because $h^{*}\circ\phi\geq(h\circ\phi)^{*}$, and (iii) $T^{**}$ depends on $\omega\in\Omega_{y,x}^{\infty}\times\Omega_{d}\times\Omega_{\varepsilon}^{\infty}$ only through the functions that are relevant for the computation of $\phi$. Moreover, (iv) \textit{if $T:=h\circ\phi$ does not depend on $\mathbf{S}_{N}$, then $T^{**}=T^{*}$ almost surely}. To prove this claim it will be enough to show that $T^{**}\leq T^{*}$ almost surely. To this aim note first that if a function $T:=h\circ\phi$ does not depend on $\mathbf{S}_{N}$, then $T$ can be written as $T:=h'\circ\phi'$ for some $h':\Omega_{y,x}^{N}\times\Omega_{\varepsilon}^{N}\mapsto\mathbb{R}$, where $\phi'$ maps the sample point $\omega\in\Omega_{y,x}^{\infty}\times\Omega_{d}\times\Omega_{\varepsilon}^{\infty}$ to the vector 
\[\left(\mathbf{Y}_{N},\mathbf{X}_{N}, \varepsilon_{1},\varepsilon_{2},\dots,\varepsilon_{N}\right).\]
Since $\phi'$ is a coordinate projection on a product probability space with product measure, $T^{*}=(h'\circ\phi')^{*}=h'^{*}\circ\phi'$ for $h'^{*}$ a measurable $(P_{y,x}^{\infty}\times P_{d}\times P_{\varepsilon}^{\infty})\circ\phi'^{-1}$-cover of $h'$ (see Lemma 1.2.5 on page 10 in \cite{vdVW}). On the other hand, $T$ can also be written as $T=h'\circ\phi''\circ\phi$ with $\phi'':\Omega_{y,x}^{N}\times\{0,1\}^{N}\times\Omega_{\varepsilon}^{N}\mapsto\Omega_{y,x}^{N}\times\Omega_{\varepsilon}^{N}$ defined in obvious way and with the same function $h'$ as before. Thus,
\[T^{**}:=(h'\circ\phi'')^{*}\circ\phi\leq h'^{*}\circ\phi''\circ\phi=h'^{*}\circ\phi'=T^{*}\quad\text{a.s.},\]
where the second-last equality follows from $\phi'=\phi''\circ\phi$.

\section{Empirical process theory for Poisson sample designs}\label{empirical_process_theory_POISSON}

Now, assume that $\{\mathbf{S}_{N}\}_{N=1}^{\infty}$ is a sequence of vectors of sample inclusion indicators corresponding to a sequence of measurable Poisson sampling designs, and let $P_{y}$ denote the marginal distribution common to the $\mathcal{Y}$-valued random variables $Y_{i}$, $i=1,2,\dots$. In this section it will be shown that under broad conditions the sequence of HTEPs corresponding to $\{\mathbf{S}_{N}\}_{N=1}^{\infty}$ and a $P_{y}$-Donsker class $\mathcal{F}$ with $\sup\{|P_{y}f|:f\in\mathcal{F}\}<\infty$ converges weakly in $l^{\infty}(\mathcal{F})$ to a Borel measurable and tight Gaussian limit process which will be denoted by $ \mathbb{G}'$. Then, opCWC and oasCWC will be shown as well. Moreover, as a corollary to unconditional weak convergence and opCWC it will also be shown that
\[(\mathbb{G}_{N}, \mathbb{G}_{N}')\rightsquigarrow(\mathbb{G}, \mathbb{G}')\text{ in }l^{\infty}(\mathcal{F})\times l^{\infty}(\mathcal{F}),\]
where
\begin{equation}\label{classical_empirical_process}
\mathbb{G}_{N}:=\frac{1}{\sqrt{N}}\sum_{i=1}^{N}(\delta_{Y_{i}}-P_{y})
\end{equation}
is the classical empirical process based on i.i.d. observations, and where $\mathbb{G}$ is a Borel measurable and tight (in $l^{\infty}(\mathcal{F})$) $P_{y}$-Brownian Bridge which is independent from $\mathbb{G}'$. 

This section will be divided in two subsections. In the first one it will be assumed that there is a positive lower bound on the first order sample inclusion probabilities. It is convenient to treat this case separately in order to make the proofs look more obvious. The second subsection will then consider the case where the first order sample inclusion probabilities are proportional to some size variable which might take on values arbitrarily close to zero.

\subsection{Weak convergence results for the case where the first order sample inclusion probabilities are bounded away from zero}\label{HTEP_lower_bound_pi}

In order to prove any of the above weak convergence results one must first establish sufficient conditions for convergence of the finite-dimensional marginal distributions. The following lemma will take care of this issue. Its conclusion says that for every finite-dimensional vector $\mathbf{f}:=(f_{1}, f_{2},\dots, f_{r})^{\intercal}\in\mathcal{F}^{r}$ and for every $\mathbf{t}\in\mathbb{R}^{r}$
\[E_{d}\exp(i\mathbf{t}^{\intercal}\mathbb{G}'_{N}\mathbf{f})\rightarrow \exp\left(-\frac{1}{2}\mathbf{t}^{\intercal}\Sigma'(\mathbf{f})\mathbf{t}\right)\]
in probability and, under the almost sure versions of its assumptions, almost surely as well. $\Sigma'(\mathbf{f})$ indicates a positive semidefinite covariance matrix that depends on the vector $\mathbf{f}$ and on the sequence of sampling designs, and $\mathbb{G}'_{N}\mathbf{f}:=(\mathbb{G}'_{N}f_{1}, \mathbb{G}'_{N}f_{2}, \dots, \mathbb{G}'_{N}f_{r})^{\intercal}$. Note that in the statement of the lemma it will be assumed that the sampling designs are measurable so that the conditional characteristic functions are measurable as well.

\begin{lem}[Convergence of marginal distributions]\label{lemma_marginal_convergence_expectation}
Let $\{\mathbf{S}_{N}\}_{N=1}^{\infty}$ be the sequence of vectors of sample inclusion indicators corresponding to a sequence of measurable Poisson sampling designs and let $\{\underline{\mathbf{\pi}}_{N}\}_{N=1}^{\infty}$ be the corresponding sequence of first order sample inclusion probabilities. Let $\mathcal{F}$ be a class of measurable functions $f:\mathcal{Y}\mapsto\mathbb{R}$ and let $\{\mathbb{G}'_{N}\}_{N=1}^{\infty}$ be the sequence of HTEPs corresponding to $\mathcal{F}$ and $\{\mathbf{S}_{N}\}_{N=1}^{\infty}$. Assume that
\begin{itemize}
\item[A1) ] there exists a function $\Sigma':\mathcal{F}^{2}\mapsto\mathbb{R}$ such that
\[\Sigma_{N}'(f,g):=E_{d}\mathbb{G}'_{N}f\mathbb{G}'_{N}g\overset{P(as)}{\rightarrow}\Sigma'(f,g)\quad\text{ for every }f,g\in\mathcal{F};\]
\item[A2) ] for every finite-dimensional vector $\mathbf{f}:=(f_{1}, f_{2},\dots, f_{r})^{\intercal}\in\mathcal{F}^{r}$
\[\frac{1}{N}\sum_{i=1}^{N}\frac{1-\pi_{i,N}}{\pi_{i,N}}\lVert\mathbf{f}(Y_{i})\rVert^{2} I(\lVert\mathbf{f}(Y_{i})\rVert>\pi_{i,N}\sqrt{N}\epsilon)\overset{P(as)}{\rightarrow}0\text{ for every }\epsilon>0,\]
where $\lVert\mathbf{f}(Y_{i})\rVert$ is the euclidean norm of $\mathbf{f}(Y_{i}):=(f_{1}(Y_{i}), f_{2}(Y_{i}), \dots, f_{r}(Y_{i}))^{\intercal}$.
\end{itemize}
Then it follows that $\Sigma'$ is a positive semidefinite covariance function, and for every finite-dimensional $\mathbf{f}\in\mathcal{F}^{r}$ and for every $\mathbf{t}\in\mathbb{R}^{r}$
\[E_{d}\exp(i\mathbf{t}^{\intercal}\mathbb{G}'_{N}\mathbf{f})\overset{P(as)}{\rightarrow} \exp\left(-\frac{1}{2}\mathbf{t}^{\intercal}\Sigma'(\mathbf{f})\mathbf{t}\right),\]
where $\Sigma'(\mathbf{f})$ is the covariance matrix whose elements are given by $\Sigma'_{(ij)}(\mathbf{f}):=\Sigma'(f_{i},f_{j})$. 
\end{lem}

\begin{proof}
The claim that $\Sigma'$ is a positive semidefinite covariance function follows immediately from the fact that $\Sigma'$ is the pointwise probability limit of the sequence of covariance functions $\{\Sigma'_{N}\}_{N=1}^{\infty}$.

Now, consider the part of the conclusion concerning the sequence of conditional characteristic functions. If $\mathbf{t}\in\mathbb{R}^{r}$ and $\mathbf{f}\in\mathcal{F}^{r}$ are such that $\mathbf{t}^{\intercal}\Sigma'(\mathbf{f})\mathbf{t}=0$, then $E_{d}|\mathbf{t}^{\intercal}\mathbb{G}'_{N}\mathbf{f}|^{2}\overset{P(as)}{\rightarrow} 0$ and the convergence result about the sequence of conditional characteristic functions is obvious in this case. So, assume that $\mathbf{t}^{\intercal}\Sigma'(\mathbf{f})\mathbf{t}>0$, and note that in this case the convergence result about the sequence of conditional characteristic functions will certainly be satisfied if a suitable probability limit version (almost sure version) of the Lindeberg condition holds. To provide an explicit expression for the latter, it will be convenient to define
\[Z_{i,N}:=\left(\frac{S_{i,N}}{\pi_{i,N}}-1\right)\mathbf{t}^{\intercal}\mathbf{f}(Y_{i}),\quad i=1,2,\dots, N,\]
and
\[q_{N}^{2}:=\sum_{i=1}^{N}E_{d}Z_{i,N}^{2}=N \mathbf{t}^{\intercal}\Sigma_{N}'(\mathbf{f})\mathbf{t},\quad N=1,2,\dots,\]
where $\Sigma_{N}'(\mathbf{f})$ is the covariance matrix whose elements are given by $\Sigma'_{N(ij)}(\mathbf{f}):=\Sigma'_{N}(f_{i},f_{j})$. Then, 
\[\mathbf{t}^{\intercal}\mathbb{G}'_{N}\mathbf{f}=\mathbb{G}'_{N}\mathbf{t}^{\intercal}\mathbf{f}=\frac{1}{\sqrt{N}}\sum_{i=1}^{N}Z_{i,N},\]
and the probability limit version (almost sure version) of the Lindeberg condition can be written as
\begin{equation}\label{P_Lindeberg}
\frac{1}{q_{N}^{2}}\sum_{i=1}^{N}E_{d}Z_{i,N}^{2}I(|Z_{i,N}|>\epsilon q_{N})\overset{P(as)}{\rightarrow} 0\quad\text{ for every }\epsilon>0.
\end{equation}
In order to prove this condition, note that
\begin{equation*}
\begin{split}
E_{d}&Z_{i,N}^{2}I(|Z_{i,N}|>\epsilon q_{N})=\\
&=\frac{(1-\pi_{i,N})^{2}}{\pi_{i,N}}\left(\mathbf{t}^{\intercal}\mathbf{f}(Y_{i})\right)^{2}I\left(\frac{1-\pi_{i,N}}{\pi_{i,N}}|\mathbf{t}^{\intercal}\mathbf{f}(Y_{i})|>\epsilon q_{N}\right)+\\
&\quad+(1-\pi_{i,N})\left(\mathbf{t}^{\intercal}\mathbf{f}(Y_{i})\right)^{2} I\left(|\mathbf{t}^{\intercal}\mathbf{f}(Y_{i})|>\epsilon q_{N}\right)\\
&\leq2\frac{1-\pi_{i,N}}{\pi_{i,N}}\lVert\mathbf{t}\rVert^{2} \lVert\mathbf{f}(Y_{i})\rVert^{2}I(\lVert\mathbf{t}\rVert \lVert\mathbf{f}(Y_{i})\rVert>\pi_{i,N}\epsilon q_{N}),
\end{split}
\end{equation*}
and that, for small enough $\eta>0$,
\[q_{N}^{2}=N \mathbf{t}^{\intercal}\Sigma_{N}'(\mathbf{f})\mathbf{t}\geq N (\mathbf{t}^{\intercal}\Sigma'(\mathbf{f})\mathbf{t}-\eta):=N C_{\eta}^{2}\rightarrow \infty\]
with probability tending to $1$ (eventually almost surely). The left side of (\ref{P_Lindeberg}) is therefore bounded by
\[\frac{2}{NC_{\eta}}\sum_{i=1}^{N}\frac{1-\pi_{i,N}}{\pi_{i,N}}\lVert\mathbf{t}\rVert^{2} \lVert\mathbf{f}(Y_{i})\rVert^{2}I(\lVert\mathbf{t}\rVert \lVert\mathbf{f}(Y_{i})\rVert>\pi_{i,N}\epsilon \sqrt{N C_{\eta}})\]
with probability tending to $1$ (eventually almost surely), and the random variable in the last display goes to zero in probability (almost surely) by assumption A2.
\end{proof}

\begin{rem}
If $\mathcal{F}$ is a $P_{y}$-Donsker class, then condition A2 will certainly be satisfied if
\begin{itemize}
\item[A2$^{*}$) ] there exists a constant $L>0$ such that 
\[\min_{1\leq i\leq N}\pi_{i,N}:=\min_{1\leq i\leq N}E_{d}S_{i,N}\geq L\]
with probability tending to $1$ (eventually almost surely).
\end{itemize}
\end{rem}

Of course, the conclusion of Lemma \ref{lemma_marginal_convergence_expectation} could also be stated by adapting the definitions of the conditional weak convergence concepts given in the previous section. This will be done in the following corollary. In its statement $G\restriction\mathcal{G}$ will denote the restriction of a generic $\mathcal{F}$-indexed stochastic process $\{Gf:f\in\mathcal{F}\}$ to a subset $\mathcal{G}\subseteq\mathcal{F}$. Note that for a \textit{finite} index set $\mathcal{G}$ the stochastic process $G\restriction\mathcal{G}$ can always be interpreted as a Borel measurable random element of $l^{\infty}(\mathcal{G})$.

\begin{cor}\label{cor_marginal_convergence}
Under the assumptions of Lemma \ref{lemma_marginal_convergence_expectation} it follows that for every finite subset $\mathcal{G}$ of $\mathcal{F}$ the random variable  
\[\sup_{g\in BL_{1}(l^{\infty}(\mathcal{G}))}\left|E_{d}g(\mathbb{G}_{N}'\restriction\mathcal{G})-Eg(\mathbb{G}'\restriction\mathcal{G})\right|\]
is measurable. Moreover, the conclusion of Lemma \ref{lemma_marginal_convergence_expectation} is equivalent to
\[\sup_{g\in BL_{1}(l^{\infty}(\mathcal{G}))}\left|E_{d}g(\mathbb{G}_{N}'\restriction\mathcal{G})-Eg(\mathbb{G}'\restriction\mathcal{G})\right|\overset{P(as)}{\rightarrow} 0\quad\text{ for every finite }\mathcal{G}\subseteq\mathcal{F},\]
where $BL_{1}(l^{\infty}(\mathcal{G}))$ is the set of all functions $h:l^{\infty}(\mathcal{G})\mapsto[0,1]$ such that $|h(z_{1}-h(z_{2})|\leq \lVert z_{1}-z_{2}\rVert_{\mathcal{G}}$ for every $z_{1},z_{2}\in l^{\infty}(\mathcal{G})$, and where $\mathbb{G}'$ is an $\mathcal{F}$-indexed zero mean Gaussian process with covariance function $\Sigma'$. 
\end{cor}

\begin{proof}
By the extension of the Portmanteau theorem in Example 1.3.5 on page 20 in \cite{vdVW} (or by L\'evy's continuity theorem) the conclusion of Lemma \ref{lemma_marginal_convergence_expectation} is equivalent to
\[E_{d}g(\mathbb{G}_{N}'\mathbf{f})\overset{P (as)}{\rightarrow} Eg(N_{r}(0,\Sigma'(\mathbf{f})))\]
for every continuous and bounded function $g:\mathbb{R}^{r}\mapsto\mathbb{R}$ and for every $\mathbf{f}\in\mathcal{F}^{r}$, where $N_{r}(0,\Sigma'(\mathbf{f}))$ denotes a random vector with $r$-dimensional centered normal distribution with covariance matrix $\Sigma'(\mathbf{f})$ (note that $E_{d}g(\mathbb{G}_{N}'\mathbf{f})$ is measurable because in Lemma \ref{lemma_marginal_convergence_expectation} it is assumed that the sampling designs are measurable). Obviously, $N_{r}(0,\Sigma'(\mathbf{f}))$ is a Borel measurable and tight random element of the metric space $\mathbb{R}^{r}$ endowed with the maximum metric. Thus, by the comments on page 73 in \cite{vdVW}, the condition in the previous display is in turn equivalent to 
\[\sup_{k\in BL_{1}(\mathbb{R}^{r})}\left|E_{d}k(\mathbb{G}_{N}'\mathbf{f})-Ek(N_{r}(0,\Sigma'(\mathbf{f})))\right|\overset{P (as)}{\rightarrow} 0,\]
where $BL_{1}(\mathbb{R}^{r})$ is the set of all functions $k:\mathbb{R}^{r}\mapsto[0,1]$ such that $|k(\mathbf{x}_{1})-k(\mathbf{x}_{2})|\leq \lVert \mathbf{x}_{1}-\mathbf{x}_{2}\rVert_{\infty}$ for every $\mathbf{x}_{1}, \mathbf{x}_{2}\in \mathbb{R}^{r}$ (here $\lVert\cdot\rVert_{\infty}$ denotes the maximum norm on $\mathbb{R}^{r}$). Since $BL_{1}(\mathbb{R}^{r})$ is separable with respect to the topology of uniform convergence on compact sets, the supremum on the left side of the previous display can be written as a supremum over a countable subset of $BL_{1}(\mathbb{R}^{r})$ which shows that it is measurable. 

To complete the proof of the corollary it is now sufficient to show that the supremum in the statement of the corollary and the supremum in the previous display are the same. To this aim, denote the elements of $\mathcal{G}$ by $f_{1}$, $f_{2}$, \dots, $f_{r}$, and consider the mapping $A:\mathbb{R}^{r}\mapsto l^{\infty}(\mathcal{G})$ that transforms the vectors $\mathbf{x}:=(x_{1}, x_{2},\dots, x_{r})^{\intercal}\in\mathbb{R}^{r}$ into the functions $z\in l^{\infty}(\mathcal{G})$ which are defined as $z(f_{i}):=x_{i}$, $i=1,2,\dots, r$. Then, $A$ is an isometry (recall that we are considering $\mathbb{R}^{r}$ endowed with the maximum norm), so that for every $g\in BL_{1}(l^{\infty}(\mathcal{G}))$ the function composition $g\circ A$ is an element of $BL_{1}(\mathbb{R}^{r})$. Since $A(\mathbb{G}_{N}'\mathbf{f})=\mathbb{G}_{N}'\restriction\mathcal{G}$, it follows that the supremum in the statement of the corollary cannot be larger than the supremum in the previous display. To obtain the opposite inequality, note that for every $k\in BL_{1}(\mathbb{R}^{r})$, the function composition $k\circ A^{-1}$ is an element of $BL_{1}(l^{\infty}(\mathcal{G}))$. 
\end{proof}

Now, for finite function classes $\mathcal{F}$ Lemma \ref{lemma_marginal_convergence_expectation} together with Corollary \ref{cor_marginal_convergence} would already establish all three the desired weak convergence results. However, to prove any of the three weak convergence results for infinite function classes $\mathcal{F}$ requires some additional work. To this aim, Lemma \ref{lemma_marginal_convergence_expectation} is very helpful because it shows that there is only one possible limit process to which the sequence of stochastic processes $\{\mathbb{G}_{N}'\}_{N=1}^{\infty}$ could possibly converge: an $\mathcal{F}$-indexed zero-mean Gaussian process with covariance function given by the function $\Sigma'$. However, for infinite function classes $\mathcal{F}$ convergence of the marginals is not enough to make sure that the sample paths of the limit process $\{\mathbb{G}'f:f\in\mathcal{F}\}$ are bounded and hence elements of $l^{\infty}(\mathcal{F})$. To prove this and the desired weak convergence results it must still be shown that the sequence of stochastic processes $\{\mathbb{G}_{N}'\}_{N=1}^{\infty}$ is (conditionally) asymptotically tight (see Theorem 1.5.4 on page 35 in \cite{vdVW}). By Theorem 1.5.7 on page 37 in \cite{vdVW} this can be done by showing that there exist a semimetric $\rho^{\bullet}$ on $\mathcal{F}$ with respect to which $\mathcal{F}$ is totally bounded and (conditionally) asymptotically equicontinuous. Since in the present paper the limit process $\{\mathbb{G}'f:f\in\mathcal{F}\}$ is supposed to be a zero-mean Gaussian process, (conditional) weak convergence in $l^{\infty}(\mathcal{F})$ can hold \textit{only if} total boundedness and (conditional) asymptotic equicontinuity hold with respect to the "$\mathbb{G}'$-intrinsic" semimetric
\begin{equation}\label{G'_intrinsic_seminorm}
\rho'(f,g):=\sqrt{\Sigma'(f-g,f-g)},\quad f,g\in\mathcal{F},
\end{equation}
(see Example 1.5.10 on pagg. 40-41 in \citep{vdVW}). But for proving the desired weak convergence results it might be more convenient to consider another semimetric in place of $\rho'$. In fact, replacing $\rho'$ with a weaker semimetric makes it easier to show total boundedness but makes it more difficult to establish asymptotic equicontinuity, while taking a stronger semimetric in place of $\rho'$ makes it harder to prove total boundedness and easier to show asymptotic equicontinuity. 

\smallskip

Now, consider first total boundedness. The following general lemma is very useful for showing total boundedness when dealing with Donsker classes. 

\begin{lem}[Total boundedness]\label{lemma_total_boundedness}
Let $Q$ be a probability measure on some measurable space, let $\mathcal{H}$ be a $Q$-Donsker class, and let $d'$ be a seminorm on $\mathcal{H}$. If 
\begin{itemize}
\item[TB) ] the expectation-centered $L_{2}(Q)$-semimetric 
\[d_{c}(f,g):=\sqrt{Q[(f-Qf)-(g-Qg)]^{2}},\quad f,g\in\mathcal{H},\]
is uniformly stronger than $d'$, i.e. $d'(f,g)\leq\zeta(d_{c}(f,g))$ for some function $\zeta:[0,\infty)\mapsto[0,\infty)$ such that $\zeta(x)\rightarrow 0$ for $x\downarrow 0$.

\end{itemize}
then $\mathcal{H}$ is totally bounded with respect to $d'$. 

If $\sup\{|P_{y}f|:f\in\mathcal{H}\}<\infty$, then condition TB can be weakened to condition
\begin{itemize}
\item[TB$_{*}$) ] the ordinary $L_{2}(Q)$-semimetric 
\[d(f,g):=\sqrt{Q(f-g)^{2}},\quad f,g\in\mathcal{H},\]
is uniformly stronger than the semimetric $d'$.
\end{itemize}
\end{lem}

\begin{proof}
If $\mathcal{H}$ is a $Q$-Donsker class, then $\mathcal{H}$ is totally bounded with respect to the seminorm $d$ (see Corollary 2.3.12 on page 115 in \cite{vdVW}), and thus, under condition TB, $\mathcal{H}$ will be totally bounded w.r.t. $d'$ as well.

The last assertion in the statement of the lemma follows from Problem 2.1.2 on page 93 in \cite{vdVW}.
\end{proof}


Next, consider asymptotic equicontinuity (hencforth AEC). Since we are going to consider the case where $\mathcal{F}$ is a $P_{y}$-Donsker class and to establish total boundedness with the aid of Lemma \ref{lemma_total_boundedness}, AEC should be established w.r.t. one of the following two semimetrics: either the expectation-centered $L_{2}(P_{y})$ semimetric
\[\rho_{c}(f,g):=\sqrt{P_{y}[(f-P_{y}f)-(g-P_{y}g)]^{2}},\quad f,g\in\mathcal{F},\]
or the ordinary $L_{2}(P_{y})$ semimetric
\[\rho(f,g):=\sqrt{P_{y}(f-g)^{2}},\quad f,g\in\mathcal{F}.\]
Of course, $\rho$ is stronger than $\rho_{c}$ and conditions which make $\{\mathbb{G}_{N}'\}_{N=1}^{\infty}$ (conditionally) AEC w.r.t. $\rho_{c}$ rather than w.r.t. $\rho$ must therefore be more stringent. Moreover, under the conditions of Lemma \ref{lemma_marginal_convergence_expectation} one should expect that there is little agreement between "$\mathbb{G}'$-intrinsic" semimetric $\rho'$ and $\rho_{c}$, while $\rho'$ and $\rho$ should usually behave quite similarly. This suggests that it is much more difficult to state general conditions which make $\{\mathbb{G}_{N}'\}_{N=1}^{\infty}$ (conditionally) AEC w.r.t. $\rho_{c}$ rather than w.r.t. $\rho$. Led by this intuition we shall therefore look for conditions which make $\{\mathbb{G}_{N}'\}_{N=1}^{\infty}$ (conditionally) AEC w.r.t. $\rho$ rather than w.r.t. $\rho_{c}$. Of course, there is a price to pay for this choice: total boundedness must be established w.r.t. to $\rho$ as well, and this can be the case only if $\sup\{|P_{y}f|:f\in\mathcal{F}\}<\infty$.

Now, before moving on to establish (conditional) AEC, it will be convenient to review the definition of the AEC concept and to give a clear definition of what \textit{conditional} AEC means. Recall that according to the definition given on page 37 in \citep{vdVW}, the sequence $\{\{\mathbb{G}_{N}'f:f\in\mathcal{F}\}\}_{N=1}^{\infty}$ is (unconditionally) AEC w.r.t. to a given semimetric $\rho^{\bullet}$ if
\[\lim_{\delta\rightarrow 0}\limsup_{N\rightarrow\infty}P^{*}\left\{\lVert\mathbb{G}_{N}'\rVert_{\mathcal{F}_{\delta}^{\bullet}}>\epsilon\right\}=0\quad\text{ for every }\epsilon>0,\]
where $P$ is the product measure $P_{y,x}^{\infty}\times P_{d}\times P_{\varepsilon}^{\infty}$, and where
\[\mathcal{F}_{\delta}^{\bullet}:=\{f-g:f,g\in\mathcal{F}\wedge \rho^{\bullet}(f,g)<\delta\},\quad \delta>0.\]
It is not difficult to show that this condition is equivalent to
\[P^{*}\left\{\lVert\mathbb{G}_{N}'\rVert_{\mathcal{F}_{\delta_{N}}^{\bullet}}>\epsilon\right\}\rightarrow 0\quad\text{ for every }\epsilon>0\text{ and for every }\delta_{N}\downarrow 0\]
which can more succinctly be written as
\[\lVert\mathbb{G}_{N}'\rVert_{\mathcal{F}_{\delta_{N}}^{\bullet}}\overset{P*}{\rightarrow}0\quad\text{ for every }\delta_{N}\downarrow 0.\]
The concept of \textit{conditional} AEC can now be defined in analogous way by requiring that
\begin{equation}\label{Def_gen_conditional_AEC}
P_{d}\left\{\lVert\mathbb{G}_{N}'\rVert_{\mathcal{F}_{\delta}^{\bullet}}>\epsilon\right\}\overset{P*(as*)}{\rightarrow} 0\quad\text{ for every }\epsilon>0\text{ and for every }\delta_{N}\downarrow 0.
\end{equation}
With respect to this definition it is worth to point out that even 
\[E_{d}\lVert\mathbb{G}_{N}'\rVert_{\mathcal{F}_{\delta_{N}}^{\bullet}}\overset{as*}{\rightarrow}0\quad\text{ for every }\delta_{N}\downarrow 0,\]
which is surely stronger than the almost sure version of conditional AEC, is \textit{apparently} not strong enough to imply the unconditional version of AEC (this is consistent with the conjecture that oasCWC does not imply unconditional weak convergence). However, in the present paper AEC will always be shown by showing that
\begin{equation}\label{AEC_conditional_expectation_versions}
E_{d}\lVert\mathbb{G}_{N}'\rVert_{\mathcal{F}_{\delta_{N}}^{\bullet}}^{**}\overset{P(as)}{\rightarrow}0\quad\text{ for every }\delta_{N}\downarrow 0
\end{equation}
for some suitable semimetric $\rho^{\bullet}$. Note that the probability version of this result is certainly stronger than unconditional AEC.

The next two lemmas are of technical nature. They will be needed to establish the two conditional expectation versions of AEC given in (\ref{AEC_conditional_expectation_versions}). The first one is an adapted version of Lemma 2.3.1 on page 108 in \cite{vdVW} (symmetrization lemma), and the second one is an adapted version of Proposition A.1.10 on page 436 in \cite{vdVW} (contraction principle).


\begin{lem}[Symmetrization inequality]\label{lemma_symmetrization_inequality}
Let $\mathcal{F}$ be an arbitrary class of measurable functions, let $\mathbf{S}_{N}$ denote the vector of sample inclusion indicators corresponding to a measurable Poisson sampling design, and let $\mathbb{G}_{N}'$ denote the HTEP corresponding to $\mathbf{S}_{N}$ and $\mathcal{F}$. Then, 
\begin{equation}\label{symmetrization_inequality}
E_{d}\lVert\mathbb{G}_{N}'\rVert_{\mathcal{F}}^{**}\leq 2 E_{\varepsilon} E_{d}\left\lVert\frac{1}{\sqrt{N}}\sum_{i=1}^{N}\varepsilon_{i}\left(\frac{S_{i,N}}{\pi_{i,N}}-1\right)\delta_{Y_{i}}\right\rVert_{\mathcal{F}}^{**}\quad\text{a.s.},
\end{equation}
where the underlying probability space and all the involved random variables are defined as described in Section \ref{Section_notation_definitions}, and where the stars on both sides of the inequality refer to the arguments of the expectations $E_{d}$ and $E_{\epsilon}E_{d}$, respectively.
\end{lem}

\begin{proof}
Let $S_{1,N}'$, $S_{2,N}'$, \dots, $S_{N,N}'$ denote the sample inclusion indicator random variables corresponding to a second Poisson sampling design which is identical to the original one but which is independent from it. To define this additional set of indicator functions recall that, by assumption, the vector $\mathbf{S}_{N}$ whose elements are the original indicator functions $S_{1,N}$, $S_{2,N}$, \dots $S_{N,N}$, is a function of the random vector $\mathbf{X}_{N}$ and of a single uniform-$[0,1]$ random variable $D$ which depends on the central coordinate of the sample points $\omega\in\Omega_{y,x}^{\infty}\times\Omega_{d}\times\Omega_{\varepsilon}^{\infty}$ only. For the definition of the new indicator functions assume WLOG that the central factor in the definition of the probability space, i.e. the factor $(\Omega_{d}, \mathcal{A}_{d}, P_{d})$, is itself a product space of the form $(\Omega_{d}, \mathcal{A}_{d}, P_{d}):=(\widetilde{\Omega}_{d},\widetilde{\mathcal{A}}_{d}, \widetilde{P}_{d})\times (\widetilde{\Omega}_{d'},\widetilde{\mathcal{A}}_{d'}, \widetilde{P}_{d'})$ where both factors on the right are identical, and assume that the random variable $D$ which determines the values taken on by $\mathbf{S}_{N}$ is actually only a function of the first coordinate of this factor space. Then, let $D'$ be an independent copy of $D$ that depends only on the second coordinate of this factor space, and define the vector $\mathbf{S}_{N}':=(S_{1,N}', S_{1,N}', \dots, S_{N,N}')$ exactly in the same way as $\mathbf{S}_{N}$ but with the random variable $D'$ in place of $D$. In what follows, if $T$ is a function that depends on $\omega\in\Omega_{y,x}^{\infty}\times\Omega_{d}\times\Omega_{\varepsilon}^{\infty}$ only through
\[\left(\mathbf{Y}_{N}, \mathbf{X}_{N}, \mathbf{S}_{N}, \mathbf{S}_{N}',\varepsilon_{1},\dots,\varepsilon_{N}\right),\]
and can thus be represented as $T=h\circ\phi'''$ for some $h:\Omega_{y,x}^{N}\times\{0,1\}^{2N}\times\Omega_{\varepsilon}^{N}$ with $\phi''':\Omega_{y,x}^{\infty}\times\Omega_{d}\times\Omega_{\varepsilon}^{\infty}\mapsto\Omega_{y,x}^{N}\times\{0,1\}^{2N}\times\Omega_{\varepsilon}^{N}$ defined in obvious way, then $T^{***}$ will be defined as $T^{***}:=h^{*}\circ\phi$ with $h^{*}$ a measurable $(P_{y,x}^{\infty}\times P_{d}\times P_{\varepsilon}^{\infty})\circ\phi'''^{-1}$-cover of $h$. Of course, $T^{***}$ has properties analogous to those of $T^{**}$ listed at the end of Section \ref{Section_notation_definitions}. In particular, (i) $T^{***}$ is measurable, (ii) $T^{***}\geq T$ almost surely, (iii) $T^{***}$ depends on $\omega\in\Omega_{y,x}^{\infty}\times\Omega_{d}\times\Omega_{\varepsilon}^{\infty}$ only through the functions that are relevant for the computation of $\phi'''$ and (iv) $T^{***}=T^{*}$ almost surely if $T$ does not depend on any of the indicator functions. Moreover, it can also be shown that $T^{***}=T^{**}$ almost surely if $T$ does not depend on $\mathbf{S}_{N}'$. The proof of the latter assertion is essentially the same as the one given at the of Section \ref{Section_notation_definitions} for showing that $T^{**}=T^{*}$ almost surely if $T$ does not depend on any of the sample inclusion indicator functions.

Now, consider the $\mathcal{F}$-indexed stochastic processes defined by
\[Z_{i}:=\frac{1}{\sqrt{N}}\left(\frac{S_{i,N}}{\pi_{i,N}}-1\right)\delta_{Y_{i}},  \quad i=1,2,\dots, N,\]
and the processes $Z_{i}'$ which are defined in the same way but with $S_{i,N}'$ in place of $S_{i,N}$. Let $\widetilde{E}_{d}$ and $\widetilde{E}_{d'}$ denote the expectations only with respect to the random variables $D$ and $D'$, respectively, with $(\mathbf{Y}_{N=\infty},\mathbf{X}_{N=\infty})$ and $\varepsilon_{1}$, $\varepsilon_{2}$, \dots kept fixed. Moreover, let $E_{d}$ denote the joint expectation w.r.t. both random variables $D$ and $D'$, still with $(\mathbf{Y}_{N=\infty},\mathbf{X}_{N=\infty})$ and $\varepsilon_{1}$, $\varepsilon_{2}$, \dots kept fixed. Note that $\widetilde{E}_{d}Z_{i}f=\widetilde{E}_{d'}Z_{i}'f=0$ for every $f\in\mathcal{F}$, so that
\begin{equation*}
\begin{split}
\lVert\mathbb{G}_{N}'\rVert_{\mathcal{F}}&:=\sup_{f\in\mathcal{F}}\left|\sum_{i=1}^{N}Z_{i}f\right|\\
&=\sup_{f\in\mathcal{F}}\left|\sum_{i=1}^{N}(Z_{i}f-\widetilde{E}_{d'}Z_{i}'f)\right|\\
&\leq \sup_{f\in\mathcal{F}}\widetilde{E}_{d'}\left|\sum_{i=1}^{N}(Z_{i}f-Z_{i}'f)\right|\\
&\leq \widetilde{E}_{d'}\left\lVert\sum_{i=1}^{N}(Z_{i}-Z_{i}')\right\rVert_{\mathcal{F}},
\end{split}
\end{equation*}
and therefore
\[\lVert\mathbb{G}_{N}'\rVert_{\mathcal{F}}^{**}=\lVert\mathbb{G}_{N}'\rVert_{\mathcal{F}}^{***}\leq\widetilde{E}_{d'}\left\lVert\sum_{i=1}^{N}(Z_{i}-Z_{i}')\right\rVert_{\mathcal{F}}^{***}\quad\text{a.s.},\]
where the stars on the far right side refer to the argument of the expectation. From this it follows that
\begin{equation}\label{presymmetrization_inequality}
\begin{split}
E_{d}\lVert\mathbb{G}_{N}'\rVert_{\mathcal{F}}^{**}&=\widetilde{E}_{d}\lVert\mathbb{G}_{N}'\rVert_{\mathcal{F}}^{***}\\
&\leq\widetilde{E}_{d}\widetilde{E}_{d'}\left\lVert\sum_{i=1}^{N}(Z_{i}-Z_{i}')\right\rVert_{\mathcal{F}}^{***}\\
&=E_{d}\left\lVert\sum_{i=1}^{N}(Z_{i}-Z_{i}')\right\rVert_{\mathcal{F}}^{***}
\end{split}
\end{equation}
almost surely, where all the stars refer to the arguments of the expectations. Now, let $h:\Omega_{y,x}^{N}\times\{0,1\}^{2N}\times\Omega_{\varepsilon}^{N}\mapsto\mathbb{R}$ be such that
\[T:=\left\lVert\sum_{i=1}^{N}(Z_{i}-Z_{i}')\right\rVert_{\mathcal{F}}=h\circ\phi''',\]
so that
\[T^{***}:=\left\lVert\sum_{i=1}^{N}(Z_{i}-Z_{i}')\right\rVert_{\mathcal{F}}^{***}:=h^{*}\circ\phi'''.\]
Moreover, let $\theta:\Omega_{y,x}^{N}\times\{0,1\}^{2N}\times\Omega_{\varepsilon}^{N}\mapsto\Omega_{y,x}^{N}\times\{0,1\}^{2N}\times\Omega_{\varepsilon}^{N}$ be the mapping that in the range of $\phi'''$ switches the positions of the indicators $S_{i,N}$ and $S_{i,N}'$ corresponding to all the indexes $i$ such that $\varepsilon_{i}=-1$. Note that $\theta$ is a measurable one-to-one mapping with measurable inverse (in fact, the inverse of $\theta$ is $\theta$ itself). Then, consider the mapping
\[T_{\varepsilon}:=\left\lVert\sum_{i=1}^{N}\varepsilon_{i}(Z_{i}-Z_{i}')\right\rVert_{\mathcal{F}}=h\circ\theta\circ\phi'''\]
and note that $T_{\varepsilon}^{***}=h^{*}\circ\theta\circ\phi'''$ almost surely with the same $h^{*}$ as in the definition of $T^{***}$ (of course, $h^{*}$ is an equivalence class). To prove this claim note that
\[T_{\varepsilon}^{***}:=(h\circ\theta)^{*}\circ\phi'''\leq h^{*}\circ\theta\circ\phi''',\]
and that the $h^{*}$ on the right side is the same as the $h^{*}$ from the definition of $T^{***}$ since $(P_{y,x}^{\infty}\times P_{d}\times P_{\varepsilon}^{\infty})\circ\phi'''^{-1}=(P_{y,x}^{\infty}\times P_{d}\times P_{\varepsilon}^{\infty})\circ\phi'''^{-1}\circ\theta^{-1}$. However,
\[h^{*}\circ\theta\circ\phi'''=(h\circ\theta\circ\theta^{-1})^{*}\circ\theta\circ\phi'''\leq (h\circ\theta)^{*}\circ\theta^{-1}\circ\theta\circ\phi'''=T_{\varepsilon}^{***},\]
which proves that $T_{\varepsilon}^{***}=h^{*}\circ\theta\circ\phi'''$ almost surely as claimed above. From the definition of the sample inclusion indicator functions it follows now that
\[E_{d}T^{***}=E_{d}h^{*}\circ \phi'''=E_{d}h^{*}\circ \theta\circ\phi'''=E_{d}T_{\varepsilon}\quad\text{a.s.}\]
Combining this fact with (\ref{presymmetrization_inequality}) yields
\[E_{d}\lVert\mathbb{G}_{N}'\rVert_{\mathcal{F}}^{**}\leq E_{\varepsilon}E_{d}\left\lVert\sum_{i=1}^{N}\varepsilon_{i}(Z_{i}-Z_{i}')\right\rVert_{\mathcal{F}}^{***}\quad\text{a.s.}\]
Now the proof can be completed by applying the triangle inequality to the right side in order to obtain
\[E_{d}\lVert\mathbb{G}_{N}'\rVert_{\mathcal{F}}^{**}\leq 2 E_{\varepsilon} E_{d}\left\lVert\sum_{i=1}^{N}\varepsilon_{i} Z_{i}\right\rVert_{\mathcal{F}}^{**}\]
which is the desired conclusion.
\end{proof}

\begin{lem}[Contraction principle]\label{contraction_principle}
Let $\mathbf{S}_{N}$ denote the vector of sample inclusion indicators corresponding to a measurable sampling design (not necessarily a Poisson sampling design). Let $\gamma_{1}$, $\gamma_{2}$, \dots, $\gamma_{N}$ be arbitrary measurable random variables which are functions of $\mathbf{Y}_{N}$ and $\mathbf{X}_{N}$ only and such that
\[0\leq\min_{1\leq i\leq N}\gamma_{i}\leq \max_{1\leq i\leq N}\gamma_{i}\leq 1\]
with probability tending to $1$ (or eventually almost surely). Let $\mathcal{F}$ be an arbitrary index set, and for each $i=1,2,\dots, N$ let $\{G_{i}f:f\in\mathcal{F}\}$ be an $\mathcal{F}$-indexed stochastic process that is a function of $\mathbf{Y}_{N}$ and $\mathbf{X}_{N}$ only. Then,
\[E_{\varepsilon} E_{d}\left\lVert\frac{1}{\sqrt{N}}\sum_{i=1}^{N}\varepsilon_{i}S_{i,N}\gamma_{i}G_{i}\right\rVert_{\mathcal{F}}^{**}\leq E_{\varepsilon}\left\lVert\frac{1}{\sqrt{N}}\sum_{i=1}^{N}\varepsilon_{i}G_{i}\right\rVert_{\mathcal{F}}^{*}\]
with probability tending to $1$ (or eventually almost surely), where, for $\mathcal{F}$-indexed stochastic processes $G$, $\lVert G \rVert_{\mathcal{F}}:=\sup_{f\in\mathcal{F}}|Gf|$, and where the stars on both sides of the inequality refer to the arguments of the two expectations $E_{\epsilon}E_{d}$ and $E_{\epsilon}$, respectively. The underlying probability space and all the involved random variables are defined as described in Section \ref{Section_notation_definitions}.
\end{lem}

\begin{proof}
Define, for $i=1,2,\dots, N$, the $\mathcal{F}$-indexed stochastic processes
\[Z_{i}:=\begin{cases}
\sum_{j=2}^{N}\varepsilon_{j}S_{j,N}\gamma_{j}G_{j} & \text{ for }i=1,\\
\sum_{j=1}^{i-1}\varepsilon_{j}G_{j}+\sum_{j=i+1}^{N}\varepsilon_{j}S_{j,N}\gamma_{j}G_{j}& \text{ for }i=2,3,\dots, N-1,\\
\sum_{j=1}^{N-1}\varepsilon_{j}G_{j}&\text{ for }i=N.
\end{cases}\]
Note that in order to prove the inequality in the conclusion of the lemma it is sufficient to prove that \textit{all} the $N$ inequalities
\[E_{\varepsilon} E_{d}\left\lVert \varepsilon_{i}S_{i,N}\gamma_{i}G_{i}+Z_{i}\right\rVert_{\mathcal{F}}^{**}\leq E_{\varepsilon} E_{d}\left\lVert \varepsilon_{i}G_{i}+Z_{i}\right\rVert_{\mathcal{F}}^{**},\quad i=1,2,\dots, N,\]
are \textit{simultaneously} satisfied with probability tending to $1$ (or eventually almost surely). In fact, for $i=1$ the left side is the same as the left side of the inequality in the conclusion of the lemma, while for $i=N$ the right side is the same as the right side of the inequality in the conclusion of the lemma. By Fubini's theorem the $N$ inequalities in the last display can also be written as
\[E_{d}E_{\varepsilon} \left\lVert \varepsilon_{i}S_{i,N}\gamma_{i}G_{i}+Z_{i}\right\rVert_{\mathcal{F}}^{**}\leq E_{d}E_{\varepsilon} \left\lVert \varepsilon_{i}G_{i}+Z_{i}\right\rVert_{\mathcal{F}}^{**},\quad i=1,2,\dots, N.\]
To show that \textit{all} these inequalities are true with probability tending to $1$ (or eventually almost surely), note that
\[0\leq \min_{1\leq i\leq N}S_{i,N}\gamma_{i}\leq\max_{1\leq i\leq N}S_{i,N}\gamma_{i}\leq 1\]
with probability tending to one (or eventually almost surely), and note that by assumption the random variables $\gamma_{i}$ are measurable and independent from the Rademacher random variables. It follows that, for every $i=1,2,\dots, N$ simultaneously, the left side in the second-last display can be bounded by
\[E_{d}S_{i,N}\gamma_{i}E_{\varepsilon}\left\lVert \varepsilon_{i}G_{i}+Z_{i}\right\rVert_{\mathcal{F}}^{**}+ E_{d}(1-S_{i,N}\gamma_{i})E_{\varepsilon}\left\lVert Z_{i}\right\rVert_{\mathcal{F}}^{**}\]
with probability tending to $1$ (or eventually almost surely). Now, note that by Fubini's theorem the joint expectation $E_{\varepsilon}$ can be replaced by an iterated expectation of the form $E_{\varepsilon_{i_{1}}}E_{\varepsilon_{i_{2}}}\cdots E_{\varepsilon_{i_{N}}}$, where $i_{1}$, $i_{2}$, \dots, $i_{N}$ is an arbitrary permutation of the natural numbers $i=1,2,\dots, N$, and where $E_{\varepsilon_{i}}$ denotes expectation with respect to the Rademacher random variable $\varepsilon_{i}$ with all other variables kept fixed. Since $E_{\varepsilon_{i}}\varepsilon_{i}=0$ for every $i=1,2,\dots, N$, it follows that the second term on the right in the last display can be written as
\[E_{d}(1-S_{i,N}\gamma_{i})E_{\varepsilon}\left\lVert G_{i}E_{\varepsilon_{i}}\varepsilon_{i}+Z_{i}\right\rVert_{\mathcal{F}}^{**},\]
which, by Jensen's inequality, is bounded by
\[E_{d}(1-S_{i,N}\gamma_{i})E_{\varepsilon}\left\lVert G_{i}\varepsilon_{i}+Z_{i}\right\rVert_{\mathcal{F}}^{**}.\]
The conclusion of the lemma follows from this.
\end{proof}


\begin{lem}[Probability version of conditional AEC]\label{lemma_AEC_expectation_POISSON}
Let $\{\mathbf{S}_{N}\}_{N=1}^{\infty}$ be defined as in Lemma \ref{lemma_marginal_convergence_expectation}, let $\mathcal{F}$ be a $P_{y}$-Donsker class, and let $\{\mathbb{G}'_{N}\}_{N=1}^{\infty}$ be the sequence of HTEPs corresponding to $\mathcal{F}$ and $\{\mathbf{S}_{N}\}_{N=1}^{\infty}$. Assume that the probability versions of conditions A1 and A2$^{*}$ hold. Then it follows that
\[E_{d}\lVert \mathbb{G}_{N}'\rVert_{\mathcal{F}_{\delta_{N}}}^{**}\overset{P}{\rightarrow} 0\quad\text{ for every }\delta_{N}\downarrow 0\]
where the stars refer to the argument of the expectation $E_{d}$, and where 
\[\mathcal{F}_{\delta}:=\{f-g:f,g\in\mathcal{F}\wedge\rho(f,g)<\delta\},\quad\delta>0.\]
\end{lem}

\begin{proof}
It is enough to prove that the expectation on the right side of (\ref{symmetrization_inequality}) with $\mathcal{F}$ replaced by $\mathcal{F}_{\delta_{N}}$ goes to zero in probability. To this aim note that by the triangle inequality the latter is bounded by 
\[\frac{1}{L}E_{\varepsilon} E_{d}\left\lVert\frac{1}{\sqrt{N}}\sum_{i=1}^{N}\varepsilon_{i}\frac{L S_{i,N}}{\pi_{i,N}}\delta_{Y_{i}}\right\rVert_{\mathcal{F}_{\delta_{N}}}^{**}+E_{\varepsilon}\left\lVert\frac{1}{\sqrt{N}}\sum_{i=1}^{N}\varepsilon_{i}\delta_{Y_{i}}\right\rVert_{\mathcal{F}_{\delta_{N}}}^{*},\]
where the stars refer to the arguments of the expectations. Then, apply the contraction principle in Lemma \ref{contraction_principle} to see that
\begin{equation}\label{contraction_inequality}
\begin{split}
E_{\varepsilon} E_{d}&\left\lVert\frac{1}{\sqrt{N}}\sum_{i=1}^{N}\varepsilon_{i}\frac{L S_{i,N}}{\pi_{i,N}}\delta_{Y_{i}}\right\rVert_{\mathcal{F}_{\delta_{N}}}^{**}\leq E_{\varepsilon}\left\lVert\frac{1}{\sqrt{N}}\sum_{i=1}^{N}\varepsilon_{i}\delta_{Y_{i}}\right\rVert_{\mathcal{F}_{\delta_{N}}}^{*}
\end{split}
\end{equation}
with probability tending to $1$. Now, observe that 
\begin{equation*}
\begin{split}
EE_{\varepsilon}&\left\lVert\frac{1}{\sqrt{N}}\sum_{i=1}^{N}\varepsilon_{i}\delta_{Y_{i}}\right\rVert_{\mathcal{F}_{\delta_{N}}}^{*}=E\left\lVert\frac{1}{\sqrt{N}}\sum_{i=1}^{N}\varepsilon_{i}\delta_{Y_{i}}\right\rVert_{\mathcal{F}_{\delta_{N}}}^{*}\\
&\leq E\left\lVert\frac{1}{\sqrt{N}}\sum_{i=1}^{N}\varepsilon_{i}(\delta_{Y_{i}}-P_{y})\right\rVert_{\mathcal{F}_{\delta_{N}}}^{*}+E\left\lVert\frac{1}{\sqrt{N}}\sum_{i=1}^{N}\varepsilon_{i}P_{y}\right\rVert_{\mathcal{F}_{\delta_{N}}}^{*}\\
&=E^{*}\left\lVert\frac{1}{\sqrt{N}}\sum_{i=1}^{N}\varepsilon_{i}(\delta_{Y_{i}}-P_{y})\right\rVert_{\mathcal{F}_{\delta_{N}}}+\lVert P_{y}\rVert_{\mathcal{F}_{\delta_{N}}}E\left|\frac{1}{\sqrt{N}}\sum_{i=1}^{N}\varepsilon_{i}\right|\\
&\leq E^{*}\left\lVert\frac{1}{\sqrt{N}}\sum_{i=1}^{N}\varepsilon_{i}(\delta_{Y_{i}}-P_{y})\right\rVert_{\mathcal{F}_{\delta_{N}}}+\delta_{N},
\end{split}
\end{equation*}
and apply the first inequality in the statement of Lemma 2.3.6 on page 111 in \cite{vdVW} to see that the outer expectation in the last line is bounded by a constant multiple of
\[E^{*}\left\lVert\frac{1}{\sqrt{N}}\sum_{i=1}^{N}(\delta_{Y_{i}}-P_{y})\right\rVert_{\mathcal{F}_{\delta_{N}}}.\]
This expectation goes to zero by Corollary 2.3.12 on page 115 in \cite{vdVW} (use the fact that $\rho_{c}\leq \rho$, where $\rho_{c}(f,g):=\sqrt{P_{y}[(f-P_{y}f)-(g-P_{y}g)]^{2}}$ is the expectation-centered $L_{2}(P_{y})$ semimetric on $\mathcal{F}$). The conclusion of the lemma follows now upon an application of Markov's inequality.
\end{proof}

\begin{lem}[Almost sure version of conditional AEC]\label{lemma_almost_sure_AEC_POISSON}
Let $\{\mathbf{S}_{N}\}_{N=1}^{\infty}$, $\mathcal{F}$ and $\{\mathbb{G}_{N}'\}_{N=1}^{\infty}$ be defined as in Lemma \ref{lemma_AEC_expectation_POISSON}, and assume that the almost sure versions of conditions A1 and A2$^{*}$ hold and moreover that condition 
\begin{itemize}
\item[S) ] $E^{*}\lVert \delta_{Y_{1}}-P_{y}\rVert_{\mathcal{F}}^{2}<\infty$
\end{itemize}
is satisfied. Then it follows that
\[E_{d}\lVert \mathbb{G}_{N}'\rVert_{\mathcal{F}_{\delta_{N}}}^{**}\overset{as}{\rightarrow} 0\quad\text{ for every }\delta_{N}\downarrow 0\]
where the stars refer to the argument of the expectation $E_{d}$.
\end{lem}

\begin{proof}
Again, it will be shown that the right side of (\ref{symmetrization_inequality}) with $\mathcal{F}_{\delta_{N}}$ in place of $\mathcal{F}$ goes to zero almost surely. To this aim, go through the steps in the proof of Lemma \ref{lemma_AEC_expectation_POISSON} up to inequality (\ref{contraction_inequality}) to see that it suffices to show that the right side of (\ref{contraction_inequality}) goes to zero almost surely. Since the right side of (\ref{contraction_inequality}) is bounded by
\[E_{\varepsilon}\left\lVert\frac{1}{\sqrt{N}}\sum_{i=1}^{N}\varepsilon_{i}(\delta_{Y_{i}}-P_{y})\right\rVert_{\mathcal{F}_{\delta_{N}}}^{*}+\delta_{N},\]
it is sufficient to show that the conditional expectation in the last display goes to zero almost surely. In order to prove the latter assertion, apply the first inequality in the statement of Lemma 2.9.9 on page 185 in \cite{vdVW} with $\mathcal{F}$ replaced by $\mathcal{F}_{\delta}$ with $\delta>0$ arbitrary but fixed, and with $\xi_{i}$ and $Z_{i}$ replaced by the Rademacher random variables $\varepsilon_{i}$ and the $\mathcal{F}_{\delta}$-indexed stochastic processes $\delta_{Y_{i}}-P_{y}$, respectively. In this way it is seen that for every $\delta>0$
\begin{equation}\label{disugaglianza_as_symmetrization}
\begin{split}
\limsup_{N\rightarrow\infty}&E_{\varepsilon}\left\lVert\frac{1}{\sqrt{N}}\sum_{i=1}^{N}\varepsilon_{i}(\delta_{Y_{i}}-P_{y})\right\rVert_{\mathcal{F}_{\delta}}^{*}\leq\\
&\leq 6\sqrt{2}\limsup_{N\rightarrow\infty}E^{*}\left\lVert\frac{1}{\sqrt{N}}\sum_{i=1}^{N}\varepsilon_{i}(\delta_{Y_{i}}-P_{y})\right\rVert_{\mathcal{F}_{\delta}}\quad\text{a.s.}
\end{split}
\end{equation}
Now, in the proof of the previous lemma it has already been shown that
\[E^{*}\left\lVert\frac{1}{\sqrt{N}}\sum_{i=1}^{N}\varepsilon_{i}(\delta_{Y_{i}}-P_{y})\right\rVert_{\mathcal{F}_{\delta_{N}}}\rightarrow 0\quad\text{ for every }\delta_{N}\downarrow 0\]
which is equivalent to 
\[\lim_{\delta\rightarrow 0}\limsup_{N\rightarrow\infty}E^{*}\left\lVert\frac{1}{\sqrt{N}}\sum_{i=1}^{N}\varepsilon_{i}(\delta_{Y_{i}}-P_{y})\right\rVert_{\mathcal{F}_{\delta}}^{*}=0.\]
In combination with (\ref{disugaglianza_as_symmetrization}) this implies that
\[\lim_{\delta\rightarrow 0}\limsup_{N\rightarrow\infty}E_{\varepsilon}\left\lVert\frac{1}{\sqrt{N}}\sum_{i=1}^{N}\varepsilon_{i}(\delta_{Y_{i}}-P_{y})\right\rVert_{\mathcal{F}_{\delta}}^{*}=0\quad\text{ a.s.}\]
and this last display is equivalent to the right side of (\ref{contraction_inequality}) going to zero almost surely for arbitrary $\delta_{N}\downarrow0$.
\end{proof}

Having gone through the most difficult technical details it remains to prove the desired weak convergence results. 

\begin{thm}[Unconditional weak convergence]\label{unconditional_convergence}
Let $\{\mathbf{S}_{N}\}_{N=1}^{\infty}$ be the sequence of vectors of sample inclusion indicators corresponding to a sequence of measurable Poisson sampling designs, let $\mathcal{F}$ be a $P_{y}$-Donsker class, and let $\{\mathbb{G}'_{N}\}_{N=1}^{\infty}$ be the sequence of HTEPs corresponding to $\mathcal{F}$ and $\{\mathbf{S}_{N}\}_{N=1}^{\infty}$. 

Assume that the probability versions of conditions A1 and A2$^{*}$ hold, and assume moreover that 
\begin{itemize}
\item[A3) ] $\mathcal{F}$ has uniformly bounded mean, i.e. $\sup\{|P_{y}f|:f\in\mathcal{F}\}<\infty$.
\end{itemize}
Then it follows that 
\begin{itemize}
\item[(i) ] there exists zero-mean Gaussian process $\{\mathbb{G}'f:f\in\mathcal{F}\}$ with covariance function given by $\Sigma'$ which is a Borel measurable and tight random element of $l^{\infty}(\mathcal{F})$ such that
\[\mathbb{G}_{N}'\rightsquigarrow\mathbb{G}'\quad\text{ in }l^{\infty}(\mathcal{F});\]
\item[(ii) ] the sample paths $f\mapsto\mathbb{G}'f$ are uniformly $\rho$-continuous with probability $1$.
\end{itemize}
\end{thm}

\begin{proof}
By Theorem 1.5.4 on page 35 and Theorem 1.5.7 on page 37 in \citep{vdVW}, for proving part (i) of the conclusion it is sufficient to show that
\begin{itemize}
\item[a) ] the marginal distributions of $\{\mathbb{G}_{N}'\}_{N=1}^{\infty}$ converge weakly to the corresponding marginal distributions of $\mathbb{G}'$, which in the notation of this paper can be written as
\[\mathbb{G}_{N}'\mathbf{f}\rightsquigarrow\mathbb{G}'\mathbf{f}\text{ in }\mathbb{R}^{r}\]
for every finite dimensional vector $\mathbf{f}:=(f_{1},f_{2},\dots, f_{r})\in\mathcal{F}^{r}$;
\item[b) ] $\{\mathbb{G}_{N}'\}_{N=1}^{\infty}$ is unconditionally AEC w.r.t. the semimetric $\rho$, which in the notation of this paper can be written as
\[\lim_{\delta\rightarrow 0}\limsup_{N\rightarrow\infty}P^{*}\left\{\lVert \mathbb{G}'_{N}\rVert_{\mathcal{F}_{\delta}}>\epsilon\right\}=0\quad\text{ for every }\epsilon>0;\]
\item[c) ] $\mathcal{F}$ is totally bounded w.r.t. the semimetric $\rho$.
\end{itemize}

Now, condition a) is an immediate consequence of Lemma \ref{lemma_marginal_convergence_expectation}, and condition c) follows from assumption A3 and Lemma \ref{lemma_total_boundedness}. It remains to show that condition b) holds. To this aim note that condition b) is equivalent to 
\[P^{*}\left\{\lVert \mathbb{G}'_{N}\rVert_{\mathcal{F}_{\delta_{N}}}>\epsilon\right\}\rightarrow 0\quad\text{ for every }\delta_{N}\downarrow0,\]
and that
\begin{equation*}
\begin{split}
P^{*}\left\{\lVert \mathbb{G}'_{N}\rVert_{\mathcal{F}_{\delta_{N}}}>\epsilon\right\}&= P\left\{\lVert \mathbb{G}'_{N}\rVert_{\mathcal{F}_{\delta_{N}}}^{*}>\epsilon\right\}\\
&\leq E P_{d}\left\{\lVert \mathbb{G}'_{N}\rVert_{\mathcal{F}_{\delta_{N}}}^{*}>\epsilon\right\}.
\end{split}
\end{equation*}
Then note that the expectation in the last line goes to zero because its argument is bounded and because by Markov's inequality and by Lemma \ref{lemma_AEC_expectation_POISSON}
\[P_{d}\left\{\lVert \mathbb{G}'_{N}\rVert_{\mathcal{F}_{\delta_{N}}}^{*}>\epsilon\right\}\leq\frac{E_{d}\lVert \mathbb{G}'_{N}\rVert_{\mathcal{F}_{\delta_{N}}}^{*}}{\epsilon}\overset{P}{\rightarrow} 0.\]
This shows that condition b) is also satisfied and completes the proof of part (i) of the conclusion. Part (ii) follows now from Addendum 1.5.8 on page 37 in \citep{vdVW}.
\end{proof}

\begin{thm}[Outer probability conditional weak convergence]\label{outer_expectation_conditional_weak_convergence}
Under the assumptions of Theorem \ref{unconditional_convergence} it follows that
\[\sup_{h\in BL_{1}(l^{\infty}(\mathcal{F}))}\left|E_{d}h(\mathbb{G}_{N}')-Eh(\mathbb{G}')\right|\overset{P*}{\rightarrow}0,\]
where $BL_{1}(l^{\infty}(\mathcal{F}))$ is the set of all functions $h:l^{\infty}(\mathcal{F})\mapsto[0,1]$ such that $|h(z_{1})-h(z_{2})|\leq \lVert z_{1}-z_{2}\rVert_{\mathcal{F}}$ for every $z_{1},z_{2}\in l^{\infty}(\mathcal{F})$, and 
where $\mathbb{G}'$ is defined as in Theorem \ref{unconditional_convergence}.
\end{thm}

\begin{proof}
The proof is almost the same as the one of Theorem 2.9.6 on page 182 in \cite{vdVW}. First note 
that in the proof of Theorem \ref{unconditional_convergence} it has already been shown that $\mathcal{F}$ is totally bounded w.r.t. $\rho$, and that the two parts of the conclusion of the same theorem imply the existence of the limit process $\{\mathbb{G}'f:f\in\mathcal{F}\}$ as a Borel measurable and tight random element of $l^{\infty}(\mathcal{F})$ whose sample paths are uniformly $\rho$-continuous with probability $1$. 

Now, since the class of functions $\mathcal{F}$ is totally bounded with respect to $\rho$, there exists for every $\delta>0$ a finite $\delta$-net for $\mathcal{F}$ (here and in the rest of this proof the underlying semimetric will always be $\rho$). Let $\mathcal{G}_{\delta}$ denote the $\delta$-net corresponding to a given $\delta>0$, and for each $\delta>0$ define a corresponding mapping $\Pi_{\delta}:\mathcal{F}\mapsto\mathcal{G}_{\delta}$ by letting $\Pi_{\delta}f$ be the element of $\mathcal{G}_{\delta}$ that is closest to $f\in\mathcal{F}$ w.r.t. $\rho$. Since $\{\mathbb{G}'f:f\in\mathcal{F}\}$ is Borel measurable (in $l^{\infty}(\mathcal{F})$) and its sample paths are uniformly $\rho$-continuous almost surely, it follows that
\[\mathbb{G}'\circ\Pi_{\delta} \overset{as}{\rightarrow}\mathbb{G}'\text{ in }l^{\infty}(\mathcal{F})\quad\text{ if }\delta\downarrow 0.\]
From this it follows further that $\mathbb{G}'\circ\Pi_{\delta} \rightsquigarrow\mathbb{G}'$ in $l^{\infty}(\mathcal{F})$ for $\delta\downarrow 0$ and the latter condition is equivalent to
\begin{equation}\label{uniform_continuous_sample_paths}
\sup_{h\in BL_{1}(l^{\infty}(\mathcal{F}))}\left|Eh(\mathbb{G}'\circ\Pi_{\delta})-Eh(\mathbb{G}')\right|\rightarrow0\quad\text{ for }\delta\downarrow 0
\end{equation}
(see the comments at the top of page 73 in \cite{vdVW}).

Next, it will be shown that for every fixed $\delta>0$
\begin{equation}\label{marginal_convergence}
\sup_{h\in BL_{1}(l^{\infty}(\mathcal{F}))}\left|E_{d}h(\mathbb{G}'_{N}\circ\Pi_{\delta})-Eh(\mathbb{G}'\circ\Pi_{\delta})\right|\overset{P*}{\rightarrow}0.
\end{equation}
To this aim, define for each $\delta>0$ the mapping $A_{\delta}:l^{\infty}(\mathcal{G}_{\delta})\mapsto l^{\infty}(\mathcal{F})$ by $A_{\delta}(z):=z\circ\Pi_{\delta}$, and note that $A_{\delta}$ transforms a function $z\in l^{\infty}(\mathcal{G}_{\delta})$ into a function $z'\in l^{\infty}(\mathcal{F})$ by extending the domain from $\mathcal{G}_{\delta}$ to $\mathcal{F}$: for $f\in\mathcal{G}_{\delta}$ the new function $z'$ remains the same (in fact, $z'(f):=z(\Pi_{\delta}(f))=z(f)$), and the new function $z'$ is constant on each level set of $\Pi_{\delta}$ (since $\mathcal{G}_{\delta}$ is finite there is only a finite number of such level sets and the range of the new function $z'$ must therefore be finite as well). Then, for $h:l^{\infty}(\mathcal{F})\mapsto\mathbb{R}$ and $G$ an $\mathcal{F}$-indexed stochastic process it follows that $h(G\circ\Pi_{\delta})=h\circ A_{\delta}(G\restriction\mathcal{G}_{\delta})$. Moreover, if $h\in BL_{1}(l^{\infty}(\mathcal{F}))$, then
\[|h\circ A_{\delta}(z_{1})-h\circ A_{\delta}(z_{2})|\leq\lVert A_{\delta}(z_{1})-A_{\delta}(z_{2})\rVert_{\mathcal{F}}=\lVert z_{1}\circ\Pi_{\delta}-z_{2}\circ\Pi_{\delta}\rVert_{\mathcal{F}}=\lVert z_{1}-z_{2}\rVert_{\mathcal{G}_{\delta}},\]
and the composition $h\circ A_{\delta}$ is therefore a member of $BL_{1}(l^{\infty}(\mathcal{G}_{\delta}))$, i.e. of the set of all functions $g:l^{\infty}(\mathcal{G}_{\delta})\mapsto[0,1]$ such that $|g(z_{1})-g(z_{2})|\leq \lVert z_{1}-z_{2}\rVert_{\mathcal{G}_{\delta}}$ for every $z_{1},z_{2}\in l^{\infty}(\mathcal{G}_{\delta})$. It follows that the supremum on the left side in (\ref{marginal_convergence}) is bounded by
\[\sup_{g\in BL_{1}(l^{\infty}(\mathcal{G}_{\delta}))}\left|E_{d}g(\mathbb{G}'_{N}\restriction\mathcal{G}_{\delta})-Eg(\mathbb{G}'\restriction\mathcal{G}_{\delta})\right|,\]
which is measurable and goes to zero in probability (see Corollary \ref{cor_marginal_convergence}). This proves (\ref{marginal_convergence}). 

Finally, in order to complete the proof note that for every fixed $\delta>0$
\begin{equation}\label{application_conditional_AEC}
\begin{split}
\sup_{h\in BL_{1}(l^{\infty}(\mathcal{F}))}&\left|E_{d}h(\mathbb{G}'_{N}\circ\Pi_{\delta})-E_{d}h(\mathbb{G}'_{N})\right|\leq\\
&\leq \sup_{h\in BL_{1}(l^{\infty}(\mathcal{F}))}E_{d}\left|h(\mathbb{G}'_{N}\circ\Pi_{\delta})-h(\mathbb{G}'_{N})\right|\\
&\leq E_{d}\lVert\mathbb{G}'_{N}\circ\Pi_{\delta}-\mathbb{G}'_{N}\rVert_{\mathcal{F}}\\
&\leq E_{d}\lVert\mathbb{G}'_{N}\rVert_{\mathcal{F}_{\delta}},
\end{split}
\end{equation}
Combine (\ref{uniform_continuous_sample_paths}) and (\ref{marginal_convergence}) with (\ref{application_conditional_AEC}) to conclude that for every fixed $\delta>0$
\begin{equation}\label{last_step_conditional_weak_convergence_proof}
\sup_{h\in BL_{1}}\left|E_{d}h(\mathbb{G}'_{N})-Eh(\mathbb{G}')\right|\leq r(\delta)+R_{N}(\delta)+E_{d}\lVert\mathbb{G}'_{N}\rVert_{\mathcal{F}_{\delta}},
\end{equation}
where $r(\delta)$ is a deterministic function of $\delta$ which goes to zero as $\delta\downarrow 0$, and where $R_{N}(\delta)=o_{P}(1)$ for every fixed $\delta>0$. From this and from Lemma \ref{lemma_AEC_expectation_POISSON} it follows that there exists a sequence $\delta_{N}\downarrow 0$ such that
\[r(\delta_{N})+R_{N}(\delta_{N})+E_{d}\lVert\mathbb{G}'_{N}\rVert_{\mathcal{F}_{\delta_{N}}}\overset{P*}{\rightarrow} 0\]
which completes the proof.
\end{proof}

\begin{cor}[Joint weak convergence]\label{corollary_joint_weak_convergence}
Under the assumptions of Theorem \ref{unconditional_convergence} it follows that
\[(\mathbb{G}_{N}, \mathbb{G}_{N}')\rightsquigarrow(\mathbb{G}, \mathbb{G}')\text{ in }l^{\infty}(\mathcal{F})\times l^{\infty}(\mathcal{F}),\]
where $\mathbb{G}_{N}'$ and $\mathbb{G}'$ are defined as in Theorem \ref{unconditional_convergence}, $\mathbb{G}_{N}$ is the classical empirical process defined in (\ref{classical_empirical_process}), and where $\mathbb{G}$ is a Borel measurable and tight  $P_{y}$-Brownian Bridge which is independent from $\mathbb{G}'$.
\end{cor}

\begin{proof}
By Example 1.4.6 on page 31 in \cite{vdVW} it suffices to show that $\mathbb{G}_{N}$ and $\mathbb{G}_{N}'$ are asymptotically independent, i.e. that
\[E^{*}f(\mathbb{G}_{N})g(\mathbb{G}_{N}')-E^{*}f(\mathbb{G}_{N}) E^{*}g(\mathbb{G}_{N}')\rightarrow 0\]
for every $f,g\in BL_{1}(l^{\infty}(\mathcal{F}))$. To this aim note that
\[E^{*}f(\mathbb{G}_{N}) E^{*}g(\mathbb{G}_{N}')\rightarrow Ef(\mathbb{G}) Eg(\mathbb{G}')\]
by Theorem \ref{unconditional_convergence} and because, by assumption, $\mathcal{F}$ is a $P_{y}$-Donsker class. To prove the corollary it must hence be shown that
\[E^{*}f(\mathbb{G}_{N})g(\mathbb{G}_{N}')\rightarrow Ef(\mathbb{G}) Eg(\mathbb{G}'),\]
which is equivalent to 
\[E^{*}[f(\mathbb{G}_{N})g(\mathbb{G}_{N}')-Ef(\mathbb{G}) Eg(\mathbb{G}')]\rightarrow0.\]
To this aim note that the argument of the outer expectation can be written as
\[f(\mathbb{G}_{N})[g(\mathbb{G}_{N}')-Eg(\mathbb{G}')]+Eg(\mathbb{G}')[f(\mathbb{G}_{N})-Ef(\mathbb{G})]:=S_{N}+T_{N},\]
and that the outer expectation in the second-last display is therefore bounded from above by $E^{*}S_{N}+E^{*}T_{N}$, and bounded from below by $E_{*}S_{N}+E_{*}T_{N}$. Now, from the definition of $\mathbb{G}_{N}\rightsquigarrow\mathbb{G}$ in $l^{\infty}(\mathcal{F})$ it follows immediately that $E^{*}T_{N}$ and $E_{*}T_{N}=-E^{*}(-T_{N})$ go both to zero. Thus, it remains to show that also $E^{*}S_{N}$ and $E_{*}S_{N}$ go both to zero. Since $E^{*}S_{N}\geq E_{*}S_{N}=-E^{*}(-S_{N})$, it suffices to show that $E^{*}S_{N}$ and $E^{*}(-S_{N})$ are bounded from above by two sequences which go both to zero. So consider first $E^{*}S_{N}$. Note that
\begin{equation*}
\begin{split}
E^{*}S_{N}&:=E^{*}f(\mathbb{G}_{N})[g(\mathbb{G}_{N}')-Eg(\mathbb{G}')]\\
&\leq E^{*}f(\mathbb{G}_{N})[g(\mathbb{G}_{N}')-E_{d}g_{*}(\mathbb{G}_{N}')+E_{d}g(\mathbb{G}_{N}')-Eg(\mathbb{G}')]\\
&\leq E^{*}f(\mathbb{G}_{N})[g(\mathbb{G}_{N}')-E_{d}g_{*}(\mathbb{G}_{N}')]+E^{*}f(\mathbb{G}_{N})[E_{d}g(\mathbb{G}_{N}')-Eg(\mathbb{G}')]\\
\end{split}
\end{equation*}
and observe that the second term in the last line goes to zero by Theorem \ref{outer_expectation_conditional_weak_convergence}. As for the first term, note that 
\begin{equation*}
\begin{split}
E^{*}f(\mathbb{G}_{N})[g(\mathbb{G}_{N}')-E_{d}g_{*}(\mathbb{G}_{N}')]&\leq Ef^{*}(\mathbb{G}_{N})[g(\mathbb{G}_{N}')-E_{d}g_{*}(\mathbb{G}_{N}')]^{*}\\
&= Ef^{*}(\mathbb{G}_{N})[g^{*}(\mathbb{G}_{N}')-E_{d}g_{*}(\mathbb{G}_{N}')]\\
&= Ef^{*}(\mathbb{G}_{N})E_{d}[g^{*}(\mathbb{G}_{N}')-E_{d}g_{*}(\mathbb{G}_{N}')]\\
&= Ef^{*}(\mathbb{G}_{N})[E_{d} g^{*}(\mathbb{G}_{N}')-E_{d}g_{*}(\mathbb{G}_{N}')]\\
&\leq E|E_{d}g^{*}(\mathbb{G}_{N}')-E_{d}g_{*}(\mathbb{G}_{N}')|\\
&= E E_{d}[g^{*}(\mathbb{G}_{N}')-g_{*}(\mathbb{G}_{N}')]\\
&= E [g^{*}(\mathbb{G}_{N}')-g_{*}(\mathbb{G}_{N}')]
\end{split}
\end{equation*}
and that the expectation in the last line goes to zero because $\mathbb{G}_{N}'\rightsquigarrow \mathbb{G}'$ in $l^{\infty}(\mathcal{F})$ by Theorem \ref{unconditional_convergence}, and because $\mathbb{G}_{N}'$ is therefore asymptotically measurable (see Lemma 1.3.8 on page 21 in \cite{vdVW}). Similar methods can be applied to show that the sequence $E^{*}(-S_{N})$ goes to zero as well.
\end{proof}

\begin{thm}[Outer almost sure conditional weak convergence]\label{outer_almost_sure_conditional_weak_convergence}
Let $\{\mathbf{S}_{N}\}_{N=1}^{\infty}$, $\mathcal{F}$ and $\{\mathbb{G}'_{N}\}_{N=1}^{\infty}$ be defined as in Theorem \ref{unconditional_convergence}. Assume that the almost sure versions of conditions A1 and A2$^{*}$ are satisfied, and assume moreover that conditions A3 and S are satisfied as well. Then,
\[\sup_{h\in BL_{1}(l^{\infty}(\mathcal{F}))}\left|E_{d}h(\mathbb{G}_{N}')-Eh(\mathbb{G}')\right|\overset{as*}{\rightarrow}0\]
where $\mathbb{G}'$ is defined as in Theorem \ref{unconditional_convergence}.
\end{thm}

\begin{proof}
The proof is the same as the proof of Theorem \ref{outer_expectation_conditional_weak_convergence}. In fact, under the conditions of the present theorem all occurrences of convergence in (outer) probability can be strengthened to (outer) almost sure convergence (at the last step of the proof Lemma \ref{lemma_almost_sure_AEC_POISSON} must be applied instead of Lemma \ref{lemma_AEC_expectation_POISSON}).
\end{proof}

\begin{rem}\label{rem_dimostrazione_generale}
The proof of Theorem \ref{outer_expectation_conditional_weak_convergence} (Theorem \ref{outer_almost_sure_conditional_weak_convergence}) can be strengthened in order to apply to any sequence of sampling designs (not necessarily Poisson sampling designs) for which it can be shown that 
\begin{itemize}
\item[(i) ] there exists a function $\Sigma':\mathcal{F}^{2}\mapsto\mathbb{R}$ such that
\[\Sigma_{N}':=E_{d}\mathbb{G}_{N}'f\mathbb{G}_{N}'g\overset{P(as)}{\rightarrow}\Sigma'\]
and such that
\begin{equation*}
\begin{split}
&E_{d}\exp(i\mathbf{t}^{\intercal}\mathbb{G}'_{N}\mathbf{f})\overset{P(as)}{\rightarrow} \exp\left(-\frac{1}{2}\mathbf{t}^{\intercal}\Sigma'(\mathbf{f})\mathbf{t}\right)\\
&\text{ for every vector $\mathbf{f}\in\mathcal{F}^{r}$ and for every $\mathbf{t}\in\mathbb{R}^{r}$, $r=1,2,\dots$,}
\end{split}
\end{equation*}
where $\Sigma'(\mathbf{f})$ is the covariance matrix whose elements are given by $\Sigma'_{(ij)}(\mathbf{f}):=\Sigma'(f_{i},f_{j})$; 
\item[(ii) ] there exists a semimetric $\rho^{\bullet}$ on $\mathcal{F}$ such that conditional AEC as defined in (\ref{Def_gen_conditional_AEC}) holds;
\item[(iii) ] $\mathcal{F}$ is totally bounded with respect to the semimetric $\rho^{\bullet}$.
\end{itemize}
\end{rem}

\subsection{Weak convergence results for the case where the first order sample inclusion probabilities are proportional to some size variable}\label{subsection_case_BBB}

For many applications the lower bound on the first order sample inclusion probabilities in assumption A2$^{*}$ is too restrictive. In fact, in many applications the first order sample inclusion probabilities are proportional to some size variable which can take on arbitrarily small values. To accomodate such cases in this subsection it will be assumed that the first order sample inclusion probabilities are defined as
\begin{equation}\label{first_order_Poisson_special_case}
\pi_{i,N}:=\min\left\{c_{N}(X_{1}, X_{2},\dots, X_{N}) \frac{w(X_{i})}{\sum_{j=1}^{N}w(X_{j})}; 1\right\},\quad i=1,2,\dots, N,
\end{equation}
where $w:\mathcal{X}\mapsto(0,\infty)$ and where $c_{N}:\mathcal{X}^{N}\mapsto(0,\infty)$ is a function which makes sure that the expected sample size equals the value taken on by some other function $n_{N}:\mathcal{X}^{N}\mapsto[0,N]$ (in many applications $\{n_{N}\}_{N=1}^{\infty}$ is simply a deterministic sequence of positive integers), i.e. $c_{N}$ makes sure that
\begin{equation}\label{expected_sample_size}
\sum_{i=1}^{N}\pi_{i,N}:=\sum_{i=1}^{N}\min\left\{c_{N}\frac{w(X_{i})}{\sum_{j=1}^{N}w(X_{j})}; 1\right\} =n_{N}.
\end{equation} 
It is not difficult to show that the function $c_{N}$ is well defined, i.e. that for every $n_{N}\in[0,N]$ there exists a unique positive constant $c_{N}$ such that equation (\ref{expected_sample_size}) holds. Moreover, under the assumptions 
\begin{itemize}
\item[B0) ] $n_{N}:\mathcal{X}^{N}\mapsto[0,N]$ is a measurable function and the sequence of expected sample sizes $\{n_{N}\}_{N=1}^{\infty}$ is such that
\[\frac{n_{N}}{N}\overset{P(as)}{\rightarrow}\alpha\in(0,1),\]
\item[B1) ] $w:\mathcal{X}\mapsto(0,\infty)$ is a measurable function such that $Ew(X_{1})<\infty$,
\end{itemize}
it can also be shown that $c_{N}$ is measurable and that $c_{N}/N\rightarrow \theta$ in probability (almost surely), where $\theta$ is the unique (positive) constant such that
\[E\min\left\{\frac{\theta w(X_{1})}{Ew(X_{2})}; 1\right\}=\alpha.\]
The details of the proof of the latter claim are left to the reader.

As already anticipated above, in what follows the theory presented in the previous subsection will be adapted in order to accommodate the case where the first order sample inclusion probabilities are defined as in (\ref{first_order_Poisson_special_case}). Under assumptions B0 and B1 this can be done by placing restrictions on the class of functions
\[\mathcal{F}/w_{\theta}:=\{f/w_{\theta}: f\in\mathcal{F}\},\]
where $\mathcal{F}$ is the original class of interest, and where
\[w_{\theta}(X_{1}):=\min\{w(X_{1}), Ew(X_{1})/\theta\}.\]
Note that the domain of the members of the class $\mathcal{F}/w_{\theta}$ is the range of the random vectors $(Y_{i},X_{i})$ (which in this paper is assumed to be $\mathcal{Y}\times\mathcal{X}$), and that the value taken on by $f/w_{\theta}\in \mathcal{F}/w_{\theta}$ at a given realization of the random vector $(Y_{i},X_{i})$ is given by $f/w_{\theta}(Y_{i},X_{i}):=f(Y_{i})/w_{\theta}(X_{i})$.

The following lemma establishes convergence of the marginal distributions of the sequence of HTEPs in the present setup. It is analogous to Lemma \ref{lemma_marginal_convergence_expectation}.

\begin{lem}[Convergence of marginal distributions]\label{lemma_marginal_convergence_B}
Let $\{\mathbf{S}_{N}\}_{N=1}^{\infty}$ be the sequence of vectors of sample inclusion indicators corresponding to a sequence of Poisson sampling designs and assume that the first order sample inclusion probabilities are defined as in (\ref{first_order_Poisson_special_case}). Let $\mathcal{F}$ be a class of measurable functions $f:\mathcal{Y}\mapsto\mathbb{R}$, and let $\{\mathbb{G}'_{N}\}_{N=1}^{\infty}$ be the sequence of HTEPs corresponding to $\mathcal{F}$ and $\{\mathbf{S}_{N}\}_{N=1}^{\infty}$. Assume that conditions B0 and B1 are satisfied and that
\begin{itemize}
\item[B2) ] the members of $\mathcal{F}/w_{\theta}$ are square integrable, i.e. $E[f(Y_{1})/w_{\theta}(X_{1})]^{2}<\infty$ for every $f\in\mathcal{F}$.
\end{itemize}
Then, 
\[\Sigma_{N}'(f,g):=E_{d}\mathbb{G}'_{N}f\mathbb{G}'_{N}g\overset{P(as)}{\rightarrow}\Sigma'(f,g)\quad\text{ for every }f,g\in\mathcal{F},\]
where
\[\Sigma'(f,g):=Ew_{\theta}(X_{1})\left(\frac{Ew(X_{2})}{\theta}-w_{\theta}(X_{1})\right)\frac{f(Y_{1})g(Y_{1})}{w_{\theta}(X_{1})^{2}},\quad f,g\in\mathcal{F}\]
is a positive semidefinite covariance function. Moreover, for every finite-dimensional $\mathbf{f}\in\mathcal{F}^{r}$ and for every $\mathbf{t}\in\mathbb{R}^{r}$ ($r$ can be any positive integer),
\[E_{d}\exp(i\mathbf{t}^{\intercal}\mathbb{G}'_{N}\mathbf{f})\overset{P(as)}{\rightarrow} \exp\left(-\frac{1}{2}\mathbf{t}^{\intercal}\Sigma'(\mathbf{f})\mathbf{t}\right),\]
where $\Sigma'(\mathbf{f})$ is the covariance matrix whose elements are given by $\Sigma'_{(ij)}(\mathbf{f}):=\Sigma'(f_{i},f_{j})$. 
\end{lem}

\begin{proof}
The definition of $\Sigma'$ can be obtained through a straightforward limit calculation by using assumptions B0, B1, B2 and the SLLN, and the claim that $\Sigma'$ is positive semidefinite follows from the fact that it is the pointwise limit of a sequence of covariance functions. As for the other part of the conclusion, it follows from Lemma \ref{lemma_marginal_convergence_expectation} upon noting that assumptions B0, B1 and B2 imply the corresponding (in probability or almost sure) versions of assumptions A1 and A2.
\end{proof}

Next consider conditional AEC. Recall that in the proof of Lemma \ref{lemma_AEC_expectation_POISSON} the lower bound on the first order sample inclusion probabilities (i.e. assumption A2$^{*}$) and the assumption that $\mathcal{F}$ is a $P_{y}$-Donsker class played a fundamental role. In the present setting, rather than assuming that $\mathcal{F}$ is a $P_{y}$-Donsker class it will be more convenient to assume that 
\begin{itemize}
\item[B2$^{*}$) ] the class $\mathcal{F}/w_{\theta}$ is a $P_{y,x}$-Donsker class
\end{itemize}
which strengthens assumption B2. 

\begin{lem}[Probability version of conditional AEC]\label{lemma_AEC_expectation_POISSON_B}
Let $\{\mathbf{S}_{N}\}_{N=1}^{\infty}$, $\mathcal{F}$ and $\{\mathbb{G}'_{N}\}_{N=1}^{\infty}$ be defined as in Lemma \ref{lemma_marginal_convergence_B}. Assume that conditions B0 (the probability version suffices), B1 and B2$^{*}$ hold. Then it follows that
\[E_{d}\lVert \mathbb{G}_{N}'\rVert_{\mathcal{F}_{\delta_{N}}^{w}}^{**}\overset{P}{\rightarrow} 0\quad\text{ for every }\delta_{N}\downarrow 0,\]
where
\[\mathcal{F}_{\delta}^{w}:=\{f-g:f,g\in\mathcal{F}\text{ and }\rho_{w}(f,g)<\delta\},\quad\delta>0,\]
with 
\[\rho_{w}(f,g):=\sqrt{P_{y,x}[(f/w_{\theta})-(g/w_{\theta})]^{2}},\quad f,g\in\mathcal{F}.\]
\end{lem}

\begin{proof}
The proof is very similar to the proof of Lemma \ref{lemma_AEC_expectation_POISSON}. First, use Lemma \ref{lemma_symmetrization_inequality} and the triangle inequality to bound the conditional expectation $E_{d}\lVert \mathbb{G}_{N}'\rVert_{\mathcal{F}_{\delta_{N}}^{w}}^{**}$ with
\begin{equation}\label{symm_bound_B}
\begin{split}
&E_{\varepsilon} E_{d}\left\lVert\frac{1}{\sqrt{N}}\sum_{i=1}^{N}\varepsilon_{i}S_{i,N}\frac{w_{\theta}(X_{i})\sum_{j=1}^{N}w(X_{j})}{w_{N}(X_{i})c_{N}}\frac{\delta_{Y_{i}}}{w_{\theta}(X_{i})}\right\rVert_{\mathcal{F}_{\delta_{N}}^{w}}^{**}+\\
&\quad +E_{\varepsilon}\left\lVert\frac{1}{\sqrt{N}}\sum_{i=1}^{N}\varepsilon_{i}\delta_{Y_{i}}\right\rVert_{\mathcal{F}_{\delta_{N}}^{w}}^{*},
\end{split}
\end{equation}
where all the stars refer to the arguments of the expectations, and where
\begin{equation}\label{definizione_funzione_w_N}
w_{N}(X_{i}):=\min\left\{w(X_{i}),\frac{1}{c_{N}}\sum_{j=1}^{N}w(X_{j})\right\}.
\end{equation}
Note that 
\[0\leq\frac{w_{\theta}(X_{i})\sum_{j=1}^{N}w(X_{j})}{w_{N}(X_{i})c_{N}}\leq\max\left\{\frac{\sum_{j=1}^{N}w(X_{j})}{c_{N}}, \frac{1}{\theta}Ew(X_{1})\right\}:=M_{N},\]
and rewrite the first term in (\ref{symm_bound_B}) as
\[M_{N}E_{\varepsilon} E_{d}\left\lVert\frac{1}{\sqrt{N}}\sum_{i=1}^{N}\varepsilon_{i}\gamma_{i}\frac{\delta_{Y_{i}}}{w_{\theta}(X_{i})}\right\rVert_{\mathcal{F}_{\delta_{N}}^{w}}^{**}\]
with 
\[\gamma_{i}:=S_{i,N}\frac{w_{\theta}(X_{i})\sum_{j=1}^{N}w(X_{j})}{M_{N} w_{N}(X_{i})n_{N}}.\]
Since $\gamma_{i}$ takes on values in $[0,1]$, the contraction principle in Lemma \ref{contraction_principle} can be applied and the first term in (\ref{symm_bound_B}) is therefore bounded by
\[M_{N}E_{\varepsilon} \left\lVert\frac{1}{\sqrt{N}}\sum_{i=1}^{N}\varepsilon_{i}\frac{\delta_{Y_{i}}}{w_{\theta}(X_{i})}\right\rVert_{\mathcal{F}_{\delta_{N}}^{w}}^{*}.\]
The contraction principle in Lemma \ref{contraction_principle} can also be applied to the second term in (\ref{symm_bound_B}). In fact, the latter can be written as
\[\frac{Ew(X_{1})}{\theta}E_{\varepsilon}\left\lVert\frac{1}{\sqrt{N}}\sum_{i=1}^{N}\varepsilon_{i}\frac{\theta w_{\theta}(X_{i})}{Ew(X_{1})}\frac{\delta_{Y_{i}}}{w_{\theta}(X_{i})}\right\rVert_{\mathcal{F}_{\delta_{N}}^{w}}^{*},\]
and by the contraction principle this is bounded by 
\[\frac{Ew(X_{1})}{\theta} E_{\varepsilon} \left\lVert\frac{1}{\sqrt{N}}\sum_{i=1}^{N}\varepsilon_{i}\frac{\delta_{Y_{i}}}{w_{\theta}(X_{i})}\right\rVert_{\mathcal{F}_{\delta_{N}}^{w}}^{*}.\]
Since $M_{N}\rightarrow Ew(X_{1})/\theta<\infty$ in probability, the proof can now be completed by showing that
\begin{equation}\label{risultato_intermedio}
E_{\varepsilon} \left\lVert\frac{1}{\sqrt{N}}\sum_{i=1}^{N}\varepsilon_{i}\frac{\delta_{Y_{i}}}{w_{\theta}(X_{i})}\right\rVert_{\mathcal{F}_{\delta_{N}}^{w}}^{*}\overset{P}{\rightarrow} 0\quad\text{ for every }\delta_{N}\downarrow 0.
\end{equation}
To this aim use the triangle inequality to bound the left side by
\begin{equation*}\label{BBBound}
\begin{split}
&E_{\varepsilon} \left\lVert\frac{1}{\sqrt{N}}\sum_{i=1}^{N}\varepsilon_{i}Z_{i}\right\rVert_{\mathcal{F}_{\delta_{N}}^{w}}^{*}+\left\lVert P_{y,x}(f/w_{\theta})\right\rVert_{\mathcal{F}_{\delta_{N}}^{w}} E_{\varepsilon}\left|\frac{1}{\sqrt{N}}\sum_{i=1}^{N}\varepsilon_{i}\right|\leq\\
&\quad\leq E_{\varepsilon} \left\lVert\frac{1}{\sqrt{N}}\sum_{i=1}^{N}\varepsilon_{i}Z_{i}\right\rVert_{\mathcal{F}_{\delta_{N}}^{w}}^{*}+\delta_{N},
\end{split}
\end{equation*}
where
\[Z_{i}:=\frac{\delta_{Y_{i}}}{w_{\theta}(X_{i})}-P_{y,x}(f/w_{\theta}),\quad i=1,2,\dots, N.\]
Then, use the first inequality in Lemma 2.3.6 on page 111 in \cite{vdVW} to see that the expectation of the right side of the inequality in the second-last display is bounded by $\delta_{N}$ plus
\[E\left\lVert\frac{1}{\sqrt{N}}\sum_{i=1}^{N}Z_{i}\right\rVert_{\mathcal{F}_{\delta_{N}}^{w}}^{*},\]
and note that this sequence of expectations goes to zero because $\mathcal{F}/w_{\theta}$ is a $P_{y,x}$-Donsker class by assumption B2$^{*}$ (use Corollary 2.3.12 on page 115 in \cite{vdVW} and the fact that $\rho_{w}$ dominates its expectation centered counterpart). 
\end{proof}

\begin{lem}[Almost sure version of conditional AEC]\label{lemma_AEC_outer_almost_sure_POISSON_B}
Let $\{\mathbf{S}_{N}\}_{N=1}^{\infty}$, $\mathcal{F}$ and $\{\mathbb{G}'_{N}\}_{N=1}^{\infty}$ be defined as in Lemma \ref{lemma_marginal_convergence_B}. Assume that conditions B0 (the almost sure version), B1 and B2$^{*}$ hold, and assume moreover that
\begin{itemize}
\item[S') ] $E^{*}\lVert \delta_{(Y_{1},X_{1})}-P_{y,x}\rVert_{\mathcal{F}/w_{\theta}}^{2}<\infty$.
\end{itemize}
Then it follows that
\[E_{d}\lVert \mathbb{G}_{N}'\rVert_{\mathcal{F}_{\delta_{N}}^{w}}^{**}\overset{as}{\rightarrow} 0\quad\text{ for every }\delta_{N}\downarrow 0.\]
\end{lem}

\begin{proof}
Follow the steps of the proof of Lemma \ref{lemma_AEC_expectation_POISSON_B} up to display (\ref{risultato_intermedio}) (note that under the assumptions of the present lemma all occurrences of convergence in probability can be replaced by almost sure convergence) and note that in order to obtain the almost sure version of conditional AEC it is sufficient to show that
\begin{equation*}\label{risultato_intermedio_1}
E_{\varepsilon} \left\lVert\frac{1}{\sqrt{N}}\sum_{i=1}^{N}\varepsilon_{i}\frac{\delta_{Y_{i}}}{w_{\theta}(X_{i})}\right\rVert_{\mathcal{F}_{\delta_{N}}^{w}}^{*}\overset{as}{\rightarrow} 0\quad\text{ for every }\delta_{N}\downarrow 0.
\end{equation*}
This can be done by the method already seen in the proof of Lemma \ref{lemma_almost_sure_AEC_POISSON}. The details are left to the reader.
\end{proof}


\begin{rem}
Example 2.10.23 on page 200 in \cite{vdVW} shows that assumption B2$^{*}$ will be satisfied whenever (i) $E(1/w_{\theta}^{2}(X_{1}))<\infty$, (ii) $\mathcal{F}$ has an envelope function $F$ such that $P_{y,x}(F^{*}/w_{\theta})^{2}<\infty$, (iii) $\mathcal{F}$ is \textit{suitably measurable} in the sense defined below, and (iv) $\mathcal{F}$ satisfies the uniform entropy condition
\begin{equation}\label{uniform_entropy}
\int_{0}^{\infty}\sup_{Q}\sqrt{\log N(\epsilon\lVert F\rVert_{Q,2}, \mathcal{F}, L_{2}(Q))}d\epsilon<\infty.
\end{equation}
In the last display the supremum is taken over all finitely discrete probability measures $Q$ on $\mathcal{Y}$ such that $\lVert F\rVert_{Q,2}:=\int F^{2}dQ>0$. 

In the context of the present remark, "$\mathcal{F}$ is \textit{suitably measurable}" means that for every $N=1,2,\dots$ and for every $(e_{1},e_{2},\dots, e_{N})\in\{-1,1\}^{N}$ the maps
\[(\mathbf{Y}_{N},\mathbf{X}_{N})\mapsto\left\lVert\sum_{i=1}^{N}e_{i}f(Y_{i})\right\rVert_{\mathcal{F}}\]
are measurable on the completion of the product space $\prod_{i=1}^{N}(\Omega_{y,x}, \mathcal{A}_{y,x}, P_{y,x})$ (cfr. Definition 2.3.3 on page 110 in \cite{vdVW}). It is easily seen that $\mathcal{F}$ is "suitably measurable" whenever it is "pointwise measurable", i.e. whenever $\mathcal{F}$ contains a countable subset $\mathcal{G}$ such that for every $f\in\mathcal{F}$ there is a sequence $\{g_{m}\}_{m=1}^{\infty}$ of functions $g_{m}\in\mathcal{G}$ such that $f$ is the pointwise limit of $\{g_{m}\}_{m=1}^{\infty}$ (see Example 2.3.4 on page 110 in \cite{vdVW}).

\end{rem}

Finally, it remains to deal with total boundedness. As in Subsection \ref{HTEP_lower_bound_pi} this can be done with the aid of Lemma \ref{lemma_total_boundedness}. However, since in the present setting the assumption that the class of functions $\mathcal{F}$ of interest is a $P_{y}$-Donsker class has been replaced by the assumption that the class of functions $\mathcal{F}/w_{\theta}$ is a $P_{y,x}$-Donsker class (see assumption B2$^{*}$), Lemma \ref{lemma_total_boundedness} must be applied with $\mathcal{H}=\mathcal{F}/w_{\theta}$ and $Q=P_{y,x}$. In this way it is easily seen that under assumption B2$^{*}$ the class of functions $\mathcal{F}$ is totally bounded w.r.t. $\rho_{w}$ whenever $\mathcal{F}/w_{\theta}$ is a $P_{y,x}$-Donsker class with $\sup\{|P_{y,x}(f/w_{\theta})|:f\in\mathcal{F}\}<\infty$.

\smallskip

Having established sufficient conditions for marginal convergence and for AEC and total boundedness w.r.t. $\rho_{w}$, one can now proceed as in the proofs of Theorem \ref{unconditional_convergence}, Theorem \ref{outer_expectation_conditional_weak_convergence}, Corollary \ref{corollary_joint_weak_convergence} and Theorem \ref{outer_almost_sure_conditional_weak_convergence} in order to obtain the following weak convergence results:

\begin{thm}[Unconditional weak convergence]\label{unconditional_convergence_B}
Let $\{\mathbf{S}_{N}\}_{N=1}^{\infty}$, $\mathcal{F}$ and $\{\mathbb{G}'_{N}\}_{N=1}^{\infty}$ be defined as in Lemma \ref{lemma_marginal_convergence_B}. Assume that conditions B0 (the probability version suffices), B1, B2$^{*}$ and assumption
\begin{itemize}
\item[B3) ] $\mathcal{F}/w_{\theta}$ has uniformly bounded mean, i.e. $\sup\{|P_{y,x}(f/w_{\theta})|:f\in\mathcal{F}\}<\infty$
\end{itemize}
are satisfied. Then it follows that
\begin{itemize}
\item[(i) ] there exists zero-mean Gaussian process $\{\mathbb{G}'f:f\in\mathcal{F}\}$ with covariance function given by $\Sigma'$ which is a Borel measurable and tight random element of $l^{\infty}(\mathcal{F})$ such that
\[\mathbb{G}_{N}'\rightsquigarrow\mathbb{G}'\quad\text{ in }l^{\infty}(\mathcal{F});\]
\item[(ii) ] the sample paths $f\mapsto\mathbb{G}'f$ are uniformly $\rho_{w}$-continuous with probability $1$.
\end{itemize}
\end{thm}

\begin{rem}\label{F_Donsker_implication}
Note that assumptions B2$^{*}$ and B3 and the fact that the function $w_{\theta}$ is uniformly bounded imply that $\mathcal{F}$ is a $P_{y}$-Donsker class for which assumption A3 holds (see Example 2.10.10 on page 192 in \citep{vdVW}).
\end{rem}

\begin{thm}[Outer probability conditional weak convergence]\label{outer_expectation_conditional_weak_convergence_B}
Under the assumptions of Theorem \ref{unconditional_convergence_B} it follows that
\[\sup_{h\in BL_{1}(l^{\infty}(\mathcal{F}))}\left|E_{d}h(\mathbb{G}_{N}')-Eh(\mathbb{G}')\right|\overset{P*}{\rightarrow}0,\]
where $\mathbb{G}'$ is defined as in Theorem \ref{unconditional_convergence_B}.
\end{thm}

\begin{cor}[Joint weak convergence]\label{corollary_joint_weak_convergence_B}
Under the assumptions of Theorem \ref{unconditional_convergence_B} it follows that
\[(\mathbb{G}_{N}, \mathbb{G}_{N}')\rightsquigarrow(\mathbb{G}, \mathbb{G}')\text{ in }l^{\infty}(\mathcal{F})\times l^{\infty}(\mathcal{F}),\]
where $\mathbb{G}_{N}'$ and $\mathbb{G}'$ are defined as in Theorem \ref{unconditional_convergence_B}, $\mathbb{G}_{N}$ is the classical empirical process defined in (\ref{classical_empirical_process}), and where $\mathbb{G}$ is a Borel measurable and tight  $P_{y}$-Brownian Bridge which is independent from $\mathbb{G}'$.
\end{cor}

\begin{thm}[Outer almost sure conditional weak convergence]\label{outer_almost_sure_conditional_weak_convergence_B}
Let $\{\mathbf{S}_{N}\}_{N=1}^{\infty}$, $\mathcal{F}$ and $\{\mathbb{G}'_{N}\}_{N=1}^{\infty}$ be defined as in Lemma \ref{lemma_marginal_convergence_B}. Assume that conditions B0 (the almost sure version), B1, B2$^{*}$, B3 and S' are satisfied. Then,
\[\sup_{h\in BL_{1}(l^{\infty}(\mathcal{F}))}\left|E_{d}h(\mathbb{G}_{N}')-Eh(\mathbb{G}')\right|\overset{as*}{\rightarrow}0,\]
where $\mathbb{G}'$ is defined as in Theorem \ref{unconditional_convergence_B}.
\end{thm}

\section{Extensions for H{\'a}jek empirical processes}\label{Hajek_empirical_process_theory_POISSON}

The theory of the previous section can quite easily be extended to sequences of H{\'a}jek empirical processes (henceforth HEP). Given a class $\mathcal{F}$ of functions $f:\mathcal{Y}\mapsto\mathbb{R}$, the HEP is defined as
\begin{equation}\label{Def_HEP}
\mathbb{G}_{N}''f:=\sqrt{N}\left(\frac{1}{\widehat{N}}\sum_{i=1}^{N}\frac{S_{i,N}}{\pi_{i,N}}f(Y_{i})-\frac{1}{N}\sum_{i=1}^{N}f(Y_{i})\right),\quad f\in\mathcal{F},
\end{equation}
with $\widehat{N}:=\sum_{i=1}^{N}(S_{i,N}/\pi_{i,N})$ the Horvitz-Thompson estimator of the population size $N$. Note that the value taken on by $\mathbb{G}_{N}''f$ is undefined in the case where $\widehat{N}=0$. However, this will not be problem here since the assumptions in the forthcoming theory will always imply that
\begin{equation}\label{condizione_consistenza_N_hat}
P_{d}\left\{\left|\frac{\widehat{N}}{N}-1\right|>\epsilon\right\}\overset{P*(as*)}{\rightarrow} 0\quad\text{ for every }\epsilon>0
\end{equation}
which allows us to consider in place of the HEP as defined in (\ref{Def_HEP}) the closely related empirical process given by
\begin{equation}\label{HEP_approximation}
\widetilde{\mathbb{G}}_{N}''f:=\frac{1}{\sqrt{N}}\sum_{i=1}^{N}\left(\frac{S_{i,N}}{\pi_{i,N}}-1\right)[f(Y_{i})-\mathbb{P}_{y,N}f],\quad f\in\mathcal{F},
\end{equation}
where $\mathbb{P}_{y,N}:=\sum_{i=1}^{N}\delta_{Y_{i}}/N$ is the empirical measure on $\mathcal{Y}$. In order to see why the HEP can be replaced by $\widetilde{\mathbb{G}}_{N}''$ it is sufficient to observe that
\[\mathbb{G}_{N}''f-\widetilde{\mathbb{G}}_{N}''f=\left(\frac{N}{\widehat{N}}-1\right) \widetilde{\mathbb{G}}_{N}''f,\quad f\in\mathcal{F},\]
and from this and condition (\ref{condizione_consistenza_N_hat}) it follows that any one of the three weak convergence results in $l^{\infty}(\mathcal{F})$ for the sequence $\{\widetilde{\mathbb{G}}_{N}''\}_{N=1}^{\infty}$ carries over immediately to the corresponding sequence of HEPs, and viceversa. 

The following lemma establishes conditional weak convergence of the marginal distributions for the sequence $\{\widetilde{\mathbb{G}}_{N}''\}_{N=1}^{\infty}$ and hence for the corresponding sequence of HEPs. 

\begin{lem}\label{marginal_convergence_HEP}
Let $\{\mathbf{S}_{N}\}_{N=1}^{\infty}$ be the sequence of vectors of sample inclusioni indicators corresponding to a sequence of measurable Poisson sampling designs and let $\{\underline{\mathbf{\pi}}_{N}\}_{N=1}^{\infty}$ the corresponding sequence of first order sample inclusion probabilities. Let $\mathcal{F}$ be a class of measurable functions $f:\mathcal{Y}\mapsto\mathbb{R}$, and let $\{\widetilde{\mathbb{G}}''\}_{N=1}^{\infty}$ be the corresponding sequence of empirical processes defined by (\ref{HEP_approximation}). Assume that conditions
\begin{itemize}
\item[C1) ] $\mathcal{F}$ contains a constant function which is not identically equal to zero, i.e. a function $f:\mathcal{Y}\mapsto\mathbb{R}$ such that $f\equiv C$ $P_{y}$-almost surely for some constant $C\neq 0$;
\item[C2) ] $P_{y}|f|<\infty$ for every $f\in\mathcal{F}$
\end{itemize}
and conditions A1 and A2 are satisfied. Then there exists a positive definite covariance function $\Sigma'':\mathcal{F}^{2}\mapsto\mathbb{R}$ such that
\[\Sigma_{N}''(f,g):=E_{d}\widetilde{\mathbb{G}}_{N}''f\widetilde{\mathbb{G}}_{N}''g\overset{P(as)}{\rightarrow}\Sigma''(f,g),\quad f,g\in\mathcal{F},\]
and for every finite-dimensional $\mathbf{f}\in\mathcal{F}^{r}$ and for every $\mathbf{t}\in\mathbb{R}^{r}$
\begin{equation}\label{limit_chf_Hajek}
E_{d}\exp(i\mathbf{t}^{\intercal}\widetilde{\mathbb{G}}''_{N}\mathbf{f})\overset{P(as)}{\rightarrow} \exp\left(-\frac{1}{2}\mathbf{t}^{\intercal}\Sigma''(\mathbf{f})\mathbf{t}\right),
\end{equation}
where $\Sigma''(\mathbf{f})$ is the covariance matrix whose elements are given by $\Sigma''_{(ij)}(\mathbf{f}):=\Sigma''(f_{i},f_{j})$. 
\end{lem}

\begin{proof}
The proof is substantially the same as the proof of Lemma \ref{lemma_marginal_convergence_expectation}. First note that
\begin{equation*}
\begin{split}
\Sigma_{N}''(f,g)&:=E_{d}\widetilde{\mathbb{G}}_{N}''f\widetilde{\mathbb{G}}_{N}''g\\
&=\frac{1}{N}\sum_{i=1}^{N}\frac{1-\pi_{i,N}}{\pi_{i,N}}(f(Y_{i})-\mathbb{P}_{y,N}f)(g(Y_{i})-\mathbb{P}_{y,N}g)\\
&=\Sigma_{N}'(f,g)-\Sigma_{N}'(f,\mathbb{P}_{y,N}g)-\Sigma_{N}'(\mathbb{P}_{y,N}f,g)+\Sigma_{N}'(\mathbb{P}_{y,N}f,\mathbb{P}_{y,N}g),
\end{split}
\end{equation*}
and that 
\[\Sigma_{N}''(f,g)\overset{P(as)}{\rightarrow}\Sigma'(f,g)-\frac{P_{y}g}{C}\Sigma'(f,C)-\frac{P_{y}f}{C}\Sigma'(C,g)+\frac{P_{y}f P_{y}g}{C^{2}}\Sigma'(C,C)\]
by assumptions C1, C2 and A1. This shows the existence of the limit function $\Sigma''$. The fact that $\Sigma''$ is positive semidefinite follows from the fact that it is the limit of a sequence of covariance functions.

Next, consider the part of the conclusion about the sequence of conditional characteristic functions. Let $\Sigma_{N}''(\mathbf{f})$ be the covariance matrix whose elements are given by $\Sigma''_{N(ij)}(\mathbf{f}):=\Sigma''_{N}(f_{i},f_{j})$, $i,j=1,2,\dots, r$, and consider first the case where $\mathbf{t}^{\intercal}\Sigma''(\mathbf{f})\mathbf{t}=0$. In this case it follows that
\[\mathbf{t}^{\intercal}\Sigma_{N}''(\mathbf{f})\mathbf{t}\overset{P(as)}{\rightarrow}\mathbf{t}^{\intercal}\Sigma''(\mathbf{f})\mathbf{t}=0\]
and this implies (\ref{limit_chf_Hajek}). In order to prove condition (\ref{limit_chf_Hajek}) also for the case where $\mathbf{t}^{\intercal}\Sigma''(\mathbf{f})\mathbf{t}>0$, it will be enough to show that a suitable probability limit (almost sure) version of the Lindeberg condition holds. In order to give an explicit expression of that condition it will be convenient to define 
\[Z_{i,N}:=\left(\frac{S_{i,N}}{\pi_{i,N}}-1\right)\mathbf{t}^{\intercal}[\mathbf{f}(Y_{i})-\mathbb{P}_{y,N}\mathbf{f}],\quad i=1,2,\dots, N,\]
and
\[q_{N}^{2}:=\sum_{i=1}^{N}E_{d}Z_{i,N}^{2}=N \mathbf{t}^{\intercal}\Sigma_{N}''(\mathbf{f})\mathbf{t},\quad N=1,2,\dots.\]
Then, 
\[\mathbf{t}^{\intercal}\widetilde{\mathbb{G}}''_{N}\mathbf{f}=\widetilde{\mathbb{G}}''_{N}\mathbf{t}^{\intercal}\mathbf{f}=\frac{1}{\sqrt{N}}\sum_{i=1}^{N}Z_{i,N},\]
and the probability limit version (almost sure version) of the Lindeberg condition can be written as
\begin{equation}\label{P_Lindeberg_Hayek}
\frac{1}{q_{N}^{2}}\sum_{i=1}^{N}E_{d}Z_{i,N}^{2}I(|Z_{i,N}|>\epsilon q_{N})\overset{P(as)}{\rightarrow} 0\quad\text{ for every }\epsilon>0.
\end{equation}
In order to prove this condition, note that
\begin{equation*}
\begin{split}
E_{d}&Z_{i,N}^{2}I(|Z_{i,N}|>\epsilon q_{N})=\\
&=\frac{(1-\pi_{i,N})^{2}}{\pi_{i,N}}\left(\mathbf{t}^{\intercal}[\mathbf{f}(Y_{i})-\mathbb{P}_{y,N}\mathbf{f}]\right)^{2}I\left(\frac{1-\pi_{i,N}}{\pi_{i,N}}|\mathbf{t}^{\intercal}[\mathbf{f}(Y_{i})-\mathbb{P}_{y,N}\mathbf{f}]|>\epsilon q_{N}\right)+\\
&\quad+(1-\pi_{i,N})\left(\mathbf{t}^{\intercal}[\mathbf{f}(Y_{i})-\mathbb{P}_{y,N}\mathbf{f}]\right)^{2} I\left(|\mathbf{t}^{\intercal}[\mathbf{f}(Y_{i})-\mathbb{P}_{y,N}\mathbf{f}]|>\epsilon q_{N}\right)\\
&\leq2\frac{1-\pi_{i,N}}{\pi_{i,N}}\lVert\mathbf{t}\rVert^{2} \lVert\mathbf{f}(Y_{i})-\mathbb{P}_{y,N}\mathbf{f}\rVert^{2}I(\lVert\mathbf{t}\rVert \lVert\mathbf{f}(Y_{i})-\mathbb{P}_{y,N}\mathbf{f}\rVert>\pi_{i,N}\epsilon q_{N})\\
&\leq2\frac{1-\pi_{i,N}}{\pi_{i,N}}\left[\lVert\mathbf{t}\rVert^{2} \lVert\mathbf{f}(Y_{i})\rVert^{2}I(\lVert\mathbf{t}\rVert \lVert\mathbf{f}(Y_{i})\rVert>\pi_{i,N}\epsilon q_{N})+\right.\\
&\quad\quad\quad\quad\quad\quad\left.+\lVert\mathbf{t}\rVert^{2} \lVert\mathbb{P}_{y,N}\mathbf{f}\rVert^{2}I(\lVert\mathbf{t}\rVert \lVert\mathbb{P}_{y,N}\mathbf{f}\rVert>\pi_{i,N}\epsilon q_{N})\right]
\end{split}
\end{equation*}
and that, for small enough $\eta>0$,
\[q_{N}^{2}=N \mathbf{t}^{\intercal}\Sigma_{N}''(\mathbf{f})\mathbf{t}\geq N (\mathbf{t}^{\intercal}\Sigma''(\mathbf{f})\mathbf{t}-\eta):=N C_{\eta}^{2}\rightarrow \infty\]
with probability tending to $1$ (eventually almost surely). The left side of (\ref{P_Lindeberg_Hayek}) is therefore bounded by
\begin{equation*}
\begin{split}
&\frac{2}{NC_{\eta}}\left[\sum_{i=1}^{N}\frac{1-\pi_{i,N}}{\pi_{i,N}}\lVert\mathbf{t}\rVert^{2} \lVert\mathbf{f}(Y_{i})\rVert^{2}I(\lVert\mathbf{t}\rVert \lVert\mathbf{f}(Y_{i})\rVert>\pi_{i,N}\epsilon \sqrt{N C_{\eta}})+\right.\\
&\quad\quad\quad\left.+\sum_{i=1}^{N}\frac{1-\pi_{i,N}}{\pi_{i,N}}\lVert\mathbf{t}\rVert^{2} \lVert\mathbb{P}_{y,N}\mathbf{f}\rVert^{2}I(\lVert\mathbf{t}\rVert \lVert\mathbb{P}_{y,N}\mathbf{f}\rVert>\pi_{i,N}\epsilon \sqrt{N C_{\eta}})\right]
\end{split}
\end{equation*}
with probability tending to $1$ (eventually almost surely), and the random variable in the last display goes to zero in probability (almost surely) by assumptions C1, C2 and A2 (assumption C1 makes sure that assumption A2 holds for constant functions $f$ as well).
\end{proof}

\begin{rem}
Assumption C2 is certainly satisfied if $\mathcal{F}$ is a $P_{y}$-Donsker class.
\end{rem}

Now, as already seen in Subsection \ref{HTEP_lower_bound_pi}, Lemma \ref{marginal_convergence_HEP} determines uniquely the finite-dimensional distributions of a stochastic process $\{\mathbb{G}''f:f\in\mathcal{F}\}$ to which the sequence $\{\{\widetilde{\mathbb{G}}_{N}''f:f\in\mathcal{F}\}\}_{N=1}^{\infty}$ (or, equivalently, the corresponding sequence of HEPs), viewed as a sequence of random elements in $l^{\infty}(\mathcal{F})$, could possibly converge in the weak sense. However, in order to prove any one of the three desired weak convergence results for infinite function classes $\mathcal{F}$ it must still be shown that there exists a semimetric for which $\mathcal{F}$ is totally bounded and for which (conditional) AEC holds. Given its similarity with the ``$\mathbb{G}''$-intrinsic'' semimetric 
\[\rho''(f,g):=\sqrt{\Sigma''(f-g,f-g)},\quad f,g\in\mathcal{F},\]
the expectation-centered $L_{2}(P_{y})$ semimetric $\rho_{c}$ seems the most obvious choice. In fact, according to Lemma \ref{lemma_total_boundedness}, if $\mathcal{F}$ is a $P_{y}$-Donsker class, then it must be totally bounded w.r.t. $\rho_{c}$. Thus, it remains to establish (conditional) AEC w.r.t. $\rho_{c}$. To this aim, the following modified version of the symmetrization inequality in Lemma \ref{lemma_symmetrization_inequality} will be needed. Its proof is word for word same as the proof of Lemma \ref{lemma_symmetrization_inequality} after replacing $\delta_{Y_{i}}$ by $\delta_{Y_{i}}-\mathbb{P}_{y,N}$. 

\begin{lem}[Symmetrization inequality]\label{HAJEK_lemma_symmetrization_inequality}
Let $\mathcal{F}$ be an arbitrary class of measurable functions, let $\mathbf{S}_{N}$ denote the vector of sample inclusion indicators corresponding to a measurable Poisson sampling design and let $\widetilde{\mathbb{G}}_{N}''$ be the corresponding empirical process defined by (\ref{HEP_approximation}). Then, 
\begin{equation}\label{HAJEK_symmetrization_inequality}
E_{d}\lVert\widetilde{\mathbb{G}}_{N}''\rVert_{\mathcal{F}}^{**}\leq 2 E_{\varepsilon} E_{d}\left\lVert\frac{1}{\sqrt{N}}\sum_{i=1}^{N}\varepsilon_{i}\left(\frac{S_{i,N}}{\pi_{i,N}}-1\right)(\delta_{Y_{i}}-\mathbb{P}_{y,N})\right\rVert_{\mathcal{F}}^{**}\quad\text{a.s.},
\end{equation}
where the underlying probability space and all the involved random variables are defined as described in Section \ref{Section_notation_definitions}, and where the stars on both sides of the inequality refer to the arguments of the expectations $E_{d}$ and $E_{\epsilon}E_{d}$, respectively.
\end{lem}

\begin{lem}[Probability version of conditional AEC]\label{HAJEK_lemma_AEC_expectation_POISSON}
Let $\{\mathbf{S}_{N}\}_{N=1}^{\infty}$ be defined as in Lemma \ref{marginal_convergence_HEP}, let $\mathcal{F}$ be a $P_{y}$-Donsker class and let $\{\widetilde{\mathbb{G}}''_{N}\}_{N=1}^{\infty}$ be the corresponding sequence of empirical processes defined by (\ref{HEP_approximation}). Assume that condition C1 is satisfied and that the probability versions of conditions A1 and A2$^{*}$ hold as well. Then it follows that
\[E_{d}\lVert \widetilde{\mathbb{G}}_{N}''\rVert_{\mathcal{F}_{\delta_{N}}^{c}}^{**}\overset{P}{\rightarrow} 0\quad\text{ for every }\delta_{N}\downarrow 0,\]
where $\mathcal{F}_{\delta}^{c}:=\{f-g:f,g\in\mathcal{F}\wedge \rho_{c}(f,g)<\delta\}$ for $\delta>0$, and where the stars refer to the argument of the expectation $E_{d}$.
\end{lem}

\begin{proof}
The proof is essentially the same as the proof of Lemma \ref{lemma_AEC_expectation_POISSON}. It suffices to prove that the expectation on the right side of (\ref{HAJEK_symmetrization_inequality}) with $\mathcal{F}$ replaced by $\mathcal{F}_{\delta_{N}}^{c}$ goes to zero in probability. To this aim note that, by the triangle inequality, the latter is bounded by 
\[\frac{1}{L}E_{\varepsilon} E_{d}\left\lVert\frac{1}{\sqrt{N}}\sum_{i=1}^{N}\varepsilon_{i}\frac{L S_{i,N}}{\pi_{i,N}}(\delta_{Y_{i}}-\mathbb{P}_{y,N})\right\rVert_{\mathcal{F}_{\delta_{N}}^{c}}^{**}+E_{\varepsilon}\left\lVert\frac{1}{\sqrt{N}}\sum_{i=1}^{N}\varepsilon_{i}(\delta_{Y_{i}}-\mathbb{P}_{y,N})\right\rVert_{\mathcal{F}_{\delta_{N}}^{c}}^{*},\]
where the stars refer to the arguments of the expectations. Then, apply the contraction principle in Lemma \ref{contraction_principle} to see that
\begin{equation}\label{HAJEK_contraction_inequality}
\begin{split}
E_{\varepsilon} E_{d}&\left\lVert\frac{1}{\sqrt{N}}\sum_{i=1}^{N}\varepsilon_{i}\frac{L S_{i,N}}{\pi_{i,N}}(\delta_{Y_{i}}-\mathbb{P}_{y,N})\right\rVert_{\mathcal{F}_{\delta_{N}}^{c}}^{**}\leq E_{\varepsilon}\left\lVert\frac{1}{\sqrt{N}}\sum_{i=1}^{N}\varepsilon_{i}(\delta_{Y_{i}}-\mathbb{P}_{y,N})\right\rVert_{\mathcal{F}_{\delta_{N}}^{c}}^{*}
\end{split}
\end{equation}
with probability tending to $1$. Now, observe that 
\begin{equation*}
\begin{split}
EE_{\varepsilon}&\left\lVert\frac{1}{\sqrt{N}}\sum_{i=1}^{N}\varepsilon_{i}(\delta_{Y_{i}}-\mathbb{P}_{y,N})\right\rVert_{\mathcal{F}_{\delta_{N}}^{c}}^{*}=E\left\lVert\frac{1}{\sqrt{N}}\sum_{i=1}^{N}\varepsilon_{i}(\delta_{Y_{i}}-\mathbb{P}_{y,N})\right\rVert_{\mathcal{F}_{\delta_{N}}^{c}}^{*}\\
&\leq E\left\lVert\frac{1}{\sqrt{N}}\sum_{i=1}^{N}\varepsilon_{i}(\delta_{Y_{i}}-P_{y})\right\rVert_{\mathcal{F}_{\delta_{N}}^{c}}^{*}+E\left\lVert\frac{1}{\sqrt{N}}\sum_{i=1}^{N}\varepsilon_{i}(\mathbb{P}_{y,N}-P_{y})\right\rVert_{\mathcal{F}_{\delta_{N}}^{c}}^{*}\\
&=E^{*}\left\lVert\frac{1}{\sqrt{N}}\sum_{i=1}^{N}\varepsilon_{i}(\delta_{Y_{i}}-P_{y})\right\rVert_{\mathcal{F}_{\delta_{N}}^{c}}+E\lVert \mathbb{P}_{y,N}-P_{y}\rVert_{\mathcal{F}_{\delta_{N}}^{c}}^{*}E_{\varepsilon}\left|\frac{1}{\sqrt{N}}\sum_{i=1}^{N}\varepsilon_{i}\right|\\
&\leq E^{*}\left\lVert\frac{1}{\sqrt{N}}\sum_{i=1}^{N}\varepsilon_{i}(\delta_{Y_{i}}-P_{y})\right\rVert_{\mathcal{F}_{\delta_{N}}^{c}}+E^{*}\lVert \mathbb{P}_{y,N}-P_{y}\rVert_{\mathcal{F}_{\delta_{N}}^{c}},
\end{split}
\end{equation*}
and apply the first inequality in the statement of Lemma 2.3.6 on page 111 in \cite{vdVW} to see that the first outer expectation in the last line is bounded by a constant multiple of
\[E^{*}\left\lVert\frac{1}{\sqrt{N}}\sum_{i=1}^{N}(\delta_{Y_{i}}-P_{y})\right\rVert_{\mathcal{F}_{\delta_{N}}^{c}}.\]
This outer expectation goes to zero by Corollary 2.3.12 on page 115 in \cite{vdVW}, and therefore the outer expectation
\[E^{*}\lVert \mathbb{P}_{y,N}-P_{y}\rVert_{\mathcal{F}_{\delta_{N}}^{c}}=E^{*}\left\lVert\frac{1}{N}\sum_{i=1}^{N}(\delta_{Y_{i}}-P_{y})\right\rVert_{\mathcal{F}_{\delta_{N}}^{c}}\]
must go to zero as well. The conclusion of the lemma follows now upon an application of Markov's inequality.
\end{proof}

\begin{lem}[Almost sure version of conditional AEC]\label{HAJEK_lemma_almost_sure_AEC_POISSON}
Let $\{\mathbf{S}_{N}\}_{N=1}^{\infty}$, $\mathcal{F}$ and $\{\widetilde{\mathbb{G}}_{N}''\}_{N=1}^{\infty}$ be defined as in Lemma \ref{HAJEK_lemma_AEC_expectation_POISSON}, and assume that condition C1, the almost sure versions of conditions A1 and A2$^{*}$ and condition S hold. Then it follows that
\[E_{d}\lVert \widetilde{\mathbb{G}}_{N}''\rVert_{\mathcal{F}_{\delta_{N}}^{c}}^{**}\overset{as}{\rightarrow} 0\quad\text{ for every }\delta_{N}\downarrow 0\]
where the stars refer to the argument of the expectation $E_{d}$.
\end{lem}

\begin{proof}
It will be shown that the right side of (\ref{HAJEK_symmetrization_inequality}) with $\mathcal{F}_{\delta_{N}}^{c}$ in place of $\mathcal{F}$ goes to zero almost surely. To this aim, go through the steps in the proof of Lemma \ref{HAJEK_lemma_AEC_expectation_POISSON} up to inequality (\ref{HAJEK_contraction_inequality}) to see that it suffices to show that the right side of (\ref{HAJEK_contraction_inequality}) goes to zero almost surely. Since the right side of (\ref{HAJEK_contraction_inequality}) is bounded by
\[E_{\varepsilon}\left\lVert\frac{1}{\sqrt{N}}\sum_{i=1}^{N}\varepsilon_{i}(\delta_{Y_{i}}-P_{y})\right\rVert_{\mathcal{F}_{\delta_{N}}^{c}}^{*}+\left\lVert\mathbb{P}_{y,N}-P_{y}\right\rVert_{\mathcal{F}_{\delta_{N}}^{c}}^{*},\]
it suffices to show that these two terms go to zero almost surely. For the first one this can be done by using the first inequality in the statement of Lemma 2.9.9 on page 185 in \cite{vdVW} (see the proof of Lemma \ref{lemma_almost_sure_AEC_POISSON}). For the second one this follows immediately from the fact that
\[\left\lVert\mathbb{P}_{y,N}-P_{y}\right\rVert_{\mathcal{F}_{\delta_{N}}^{c}}^{*}\leq 2\left\lVert\mathbb{P}_{y,N}-P_{y}\right\rVert_{\mathcal{F}}^{*}\]
and the fact that every $P_{y}$-Donsker class is an outer almost sure $P_{y}$-Glivenko-Cantelli class (in fact, Corollary 2.3.13 on page 115 in \citep{vdVW} implies that for every $P_{y}$-Donsker class $\mathcal{F}$ the random variable $\left\lVert\delta_{Y_{i}}-P_{y}\right\rVert_{\mathcal{F}}^{*}$ has a weak second moment, and thus it follows from Lemma 2.4.5 on page 124 in \citep{vdVW} that $\left\lVert\mathbb{P}_{y,N}-P_{y}\right\rVert_{\mathcal{F}}^{*}$ converges almost surely to a constant $c$; since for every $P_{y}$-Donsker class $\mathcal{F}$ it is certainly true that $\left\lVert\mathbb{P}_{y,N}-P_{y}\right\rVert_{\mathcal{F}}^{*}\rightarrow 0$ in probability, the constant $c$ must be zero).
\end{proof}

Now, having established sufficient conditions for convergence of the marginal distributions and for total boundedness and (conditional) AEC w.r.t. $\rho_{c}$, one can proceed as in the proofs of Theorem \ref{unconditional_convergence}, Theorem \ref{outer_expectation_conditional_weak_convergence}, Corollary \ref{corollary_joint_weak_convergence} and Theorem \ref{outer_almost_sure_conditional_weak_convergence} in order to show the three desired weak convergence results for the sequence $\{\widetilde{\mathbb{G}}_{N}''\}_{N=1}^{\infty}$. However, in this section it has always been assumed that the sequence of Poisson sampling designs is measurable, and that assumptions C1 and A1 hold. These conditions imply that condition (\ref{condizione_consistenza_N_hat}) must be satisfied, and from this it follows that any one of the three desired weak convergence results about $\{\widetilde{\mathbb{G}}_{N}''\}_{N=1}^{\infty}$ carries over to the corresponding sequence of HEPs, and viceversa. Since in applications only the latter is of interest, the weak convergence results will be stated only for HEPs. 

\begin{thm}[Unconditional weak convergence]\label{HAJEK_unconditional_convergence}
Let $\{\mathbf{S}_{N}\}_{N=1}^{\infty}$ be the sequence of vectors of sample inclusion indicators corresponding to a sequence of measurable Poisson sampling designs, let $\mathcal{F}$ be a $P_{y}$-Donsker class, and let $\{\mathbb{G}''_{N}\}_{N=1}^{\infty}$ be the sequence of HEPs corresponding to $\mathcal{F}$ and $\{\mathbf{S}_{N}\}_{N=1}^{\infty}$. 

Assume that condition C1 and the probability versions of conditions A1 and A2$^{*}$ are satisfied. Then it follows that
\begin{itemize}
\item[(i) ] there exists zero-mean Gaussian process $\{\mathbb{G}''f:f\in\mathcal{F}\}$ with covariance function given by $\Sigma''$ which is a Borel measurable and tight random element of $l^{\infty}(\mathcal{F})$ such that
\[\mathbb{G}_{N}''\rightsquigarrow\mathbb{G}''\quad\text{ in }l^{\infty}(\mathcal{F});\]
\item[(ii) ] the sample paths $f\mapsto\mathbb{G}''f$ are uniformly $\rho_{c}$-continuous with probability $1$.
\end{itemize}
\end{thm}

\begin{thm}[Outer probability conditional weak convergence]\label{HAJEK_outer_expectation_conditional_weak_convergence}
Under the assumptions of Theorem \ref{HAJEK_unconditional_convergence} it follows that
\[\sup_{h\in BL_{1}(l^{\infty}(\mathcal{F}))}\left|E_{d}h(\mathbb{G}_{N}'')-Eh(\mathbb{G}'')\right|\overset{P*}{\rightarrow}0,\]
where $\mathbb{G}''$ is defined as in Theorem \ref{HAJEK_unconditional_convergence}.
\end{thm}

\begin{cor}[Joint weak convergence]\label{HAJEK_corollary_joint_weak_convergence}
Under the assumptions of Theorem \ref{HAJEK_unconditional_convergence} it follows that
\[(\mathbb{G}_{N}, \mathbb{G}_{N}'')\rightsquigarrow(\mathbb{G}, \mathbb{G}'')\text{ in }l^{\infty}(\mathcal{F})\times l^{\infty}(\mathcal{F}),\]
where $\mathbb{G}_{N}''$ and $\mathbb{G}''$ are defined as in Theorem \ref{HAJEK_unconditional_convergence}, $\mathbb{G}_{N}$ is the classical empirical process defined in (\ref{classical_empirical_process}), and where $\mathbb{G}$ is a Borel measurable and tight  $P_{y}$-Brownian Bridge which is independent from $\mathbb{G}''$.
\end{cor}

\begin{thm}[Outer almost sure conditional weak convergence]\label{HAJEK_outer_almost_sure_conditional_weak_convergence}
Let $\{\mathbf{S}_{N}\}_{N=1}^{\infty}$, $\mathcal{F}$ and $\{\mathbb{G}''_{N}\}_{N=1}^{\infty}$ be defined as in Theorem \ref{HAJEK_unconditional_convergence}. Assume that condition C1, the almost sure versions of conditions A1 and A2$^{*}$ and condition S are satisfied. Then it follows that
\[\sup_{h\in BL_{1}(l^{\infty}(\mathcal{F}))}\left|E_{d}h(\mathbb{G}_{N}'')-Eh(\mathbb{G}'')\right|\overset{as*}{\rightarrow}0.\]
where $\mathbb{G}''$ is defined as in Theorem \ref{HAJEK_unconditional_convergence}.
\end{thm}

\begin{rem}
Note that the assumptions of Theorem \ref{HAJEK_outer_almost_sure_conditional_weak_convergence} (and hence also the assumptions of Theorem \ref{HAJEK_unconditional_convergence}) are often satisfied also if $\mathcal{F}$ is a $P_{y}$-Donsker class with $\sup\{|P_{y}f|:f\in\mathcal{F}\}=\infty$, i.e. if assumption A3 fails (cfr. this with Theorem \ref{unconditional_convergence} and Theorem \ref{outer_almost_sure_conditional_weak_convergence}).
\end{rem}

Finally, consider the case where the first order sample inclusion probabilities are defined as in (\ref{first_order_Poisson_special_case}) which could give rise to arbitrarily small values. This case has already been treated for HTEP sequences in Subsection \ref{subsection_case_BBB}. Under slightly more restrictive assumptions it can be shown that the weak convergence results for HTEP sequences extend also to HEP sequences. Unfortunately, assumption B3 (which is analogous to assumption A3) cannot be dropped in general because it is not possible to show (conditional) AEC w.r.t. the expectation centered counterpart of the semimetric $\rho_{w}$. Actually, assumption B3 can be dropped if it is assumed that the two components of the $(Y_{i},X_{i})$ vectors are independent but this fact will not enter the statements of the next battery of weak convergence results.

\begin{lem}[Convergence of marginal distributions]\label{lemma_marginal_convergence_B_HEP}
Let $\{\mathbf{S}_{N}\}_{N=1}^{\infty}$ be the sequence of vectors of sample inclusion indicators corresponding to a sequence of Poisson sampling designs and assume that the first order sample inclusion probabilities are defined as in (\ref{first_order_Poisson_special_case}). Let $\mathcal{F}$ be a class of measurable functions $f:\mathcal{Y}\mapsto\mathbb{R}$ and let $\{\widetilde{\mathbb{G}}''\}_{N=1}^{\infty}$ be the corresponding sequence of empirical processes defined by (\ref{HEP_approximation}). Assume that conditions C1, B0, B1 and B2 are satisfied. Then it follows that 
\[\Sigma_{N}''(f,g):=E_{d}\widetilde{\mathbb{G}}''_{N}f\widetilde{\mathbb{G}}''_{N}g\overset{P(as)}{\rightarrow}\Sigma''(f,g)\quad\text{ for every }f,g\in\mathcal{F},\]
where
\[\Sigma''(f,g):=Ew_{\theta}(X_{1})\left(\frac{Ew(X_{2})}{\theta}-w_{\theta}(X_{1})\right)\frac{[f(Y_{1})-P_{y}f][g(Y_{1})-P_{y}g]}{w_{\theta}(X_{1})^{2}},\quad f,g\in\mathcal{F},\]
is a positive semidefinite covariance function. Moreover, for every finite-dimensional $\mathbf{f}\in\mathcal{F}^{r}$ and for every $\mathbf{t}\in\mathbb{R}^{r}$ ($r$ can be any positive integer)
\[E_{d}\exp(i\mathbf{t}^{\intercal}\widetilde{\mathbb{G}}''_{N}\mathbf{f})\overset{P(as)}{\rightarrow} \exp\left(-\frac{1}{2}\mathbf{t}^{\intercal}\Sigma''(\mathbf{f})\mathbf{t}\right),\]
where $\Sigma''(\mathbf{f})$ is the covariance matrix whose elements are given by $\Sigma''_{(ij)}(\mathbf{f}):=\Sigma''(f_{i},f_{j})$. 
\end{lem}

\begin{proof}
The definition of $\Sigma''$ can be obtained through a straightforward limit calculation by using assumptions B0, B1, B2 and the SLLN (note that  assumption C2 follows from B0, B1 and B2), and the claim that $\Sigma''$ is positive semidefinite follows from the fact that it is the pointwise limit of a sequence of covariance functions. This proves the first part of the conclusion. The second part of the conclusion can now be proved as it was done in the proof of Lemma \ref{marginal_convergence_HEP} upon noting that assumptions B0, B1 and B2 imply assumption A2.
\end{proof}

\begin{lem}[Probability version of conditional AEC]\label{lemma_AEC_expectation_POISSON_B_HEP}
Let $\{\mathbf{S}_{N}\}_{N=1}^{\infty}$, $\mathcal{F}$ and $\{\widetilde{\mathbb{G}}''_{N}\}_{N=1}^{\infty}$ be defined as in Lemma \ref{lemma_marginal_convergence_B_HEP}. Assume that conditions C1, B0 (the probability version suffices), B1, B2$^{*}$ and B3 hold. Then it follows that
\[E_{d}\lVert \widetilde{\mathbb{G}}_{N}''\rVert_{\mathcal{F}_{\delta_{N}}^{w}}^{**}\overset{P}{\rightarrow} 0\quad\text{ for every }\delta_{N}\downarrow 0.\]
\end{lem}

\begin{proof}
The proof is very similar to the proof of Lemma \ref{HAJEK_lemma_AEC_expectation_POISSON}. First, use Lemma \ref{HAJEK_lemma_symmetrization_inequality} and the triangle inequality to bound the conditional expectation $E_{d}\lVert \widetilde{\mathbb{G}}_{N}''\rVert_{\mathcal{F}_{\delta_{N}}^{w}}^{**}$ with
\begin{equation}\label{symm_bound_B_HEP}
\begin{split}
&E_{\varepsilon} E_{d}\left\lVert\frac{1}{\sqrt{N}}\sum_{i=1}^{N}\varepsilon_{i}S_{i,N}\frac{w_{\theta}(X_{i})\sum_{j=1}^{N}w(X_{j})}{w_{N}(X_{i})c_{N}}\frac{\delta_{Y_{i}}-\mathbb{P}_{y,N}}{w_{\theta}(X_{i})}\right\rVert_{\mathcal{F}_{\delta_{N}}^{w}}^{**}+\\
&\quad +E_{\varepsilon}\left\lVert\frac{1}{\sqrt{N}}\sum_{i=1}^{N}\varepsilon_{i}(\delta_{Y_{i}}-\mathbb{P}_{y,N})\right\rVert_{\mathcal{F}_{\delta_{N}}^{w}}^{*},
\end{split}
\end{equation}
where all the stars refer to the arguments of the expectations, and where $w_{N}(X_{i})$ is defined as in (\ref{definizione_funzione_w_N}). Note that 
\[0\leq\frac{w_{\theta}(X_{i})\sum_{j=1}^{N}w(X_{j})}{w_{N}(X_{i})c_{N}}\leq\max\left\{\frac{\sum_{j=1}^{N}w(X_{j})}{c_{N}}, \frac{1}{\theta}Ew(X_{1})\right\}:=M_{N},\]
and rewrite the first term in (\ref{symm_bound_B_HEP}) as
\[M_{N}E_{\varepsilon} E_{d}\left\lVert\frac{1}{\sqrt{N}}\sum_{i=1}^{N}\varepsilon_{i}\gamma_{i}\frac{\delta_{Y_{i}}-\mathbb{P}_{y,N}}{w_{\theta}(X_{i})}\right\rVert_{\mathcal{F}_{\delta_{N}}^{w}}^{**}\]
with 
\[\gamma_{i}:=S_{i,N}\frac{w_{\theta}(X_{i})\sum_{j=1}^{N}w(X_{j})}{M_{N} w_{N}(X_{i})n_{N}}.\]
Since $\gamma_{i}$ takes on values in $[0,1]$, the contraction principle in Lemma \ref{contraction_principle} can be applied and the first term in (\ref{symm_bound_B_HEP}) is therefore bounded by
\[M_{N}E_{\varepsilon} \left\lVert\frac{1}{\sqrt{N}}\sum_{i=1}^{N}\varepsilon_{i}\frac{\delta_{Y_{i}}-\mathbb{P}_{y,N}}{w_{\theta}(X_{i})}\right\rVert_{\mathcal{F}_{\delta_{N}}^{w}}^{*}.\]
The contraction principle in Lemma \ref{contraction_principle} can also be applied to the second term in (\ref{symm_bound_B_HEP}). In fact, the latter can be written as
\[\frac{Ew(X_{1})}{\theta}E_{\varepsilon}\left\lVert\frac{1}{\sqrt{N}}\sum_{i=1}^{N}\varepsilon_{i}\frac{\theta w_{\theta}(X_{i})}{Ew(X_{1})}\frac{\delta_{Y_{i}}-\mathbb{P}_{y,N}}{w_{\theta}(X_{i})}\right\rVert_{\mathcal{F}_{\delta_{N}}^{w}}^{*},\]
and by the contraction principle this is bounded by 
\[\frac{Ew(X_{1})}{\theta} E_{\varepsilon} \left\lVert\frac{1}{\sqrt{N}}\sum_{i=1}^{N}\varepsilon_{i}\frac{\delta_{Y_{i}}-\mathbb{P}_{y,N}}{w_{\theta}(X_{i})}\right\rVert_{\mathcal{F}_{\delta_{N}}^{w}}^{*}.\]
Since $M_{N}\rightarrow Ew(X_{1})/\theta<\infty$ in probability, the proof can now be completed by showing that
\begin{equation}\label{risultato_intermedio_HEP}
E_{\varepsilon} \left\lVert\frac{1}{\sqrt{N}}\sum_{i=1}^{N}\varepsilon_{i}\frac{\delta_{Y_{i}}-\mathbb{P}_{y,N}}{w_{\theta}(X_{i})}\right\rVert_{\mathcal{F}_{\delta_{N}}^{w}}^{*}\overset{P}{\rightarrow} 0\quad\text{ for every }\delta_{N}\downarrow 0.
\end{equation}
To this aim use the triangle inequality to bound the left side by
\begin{equation}\label{BBBound_HEP}
\begin{split}
&E_{\varepsilon} \left\lVert\frac{1}{\sqrt{N}}\sum_{i=1}^{N}\varepsilon_{i}Z_{i}\right\rVert_{\mathcal{F}_{\delta_{N}}^{w}}^{*}+\left\lVert P_{y,x}(f/w_{\theta})\right\rVert_{\mathcal{F}_{\delta_{N}}^{w}}^{*} E_{\varepsilon}\left|\frac{1}{\sqrt{N}}\sum_{i=1}^{N}\varepsilon_{i}\right|+\\
&+\left\lVert \mathbb{P}_{y,N}\right\rVert_{\mathcal{F}_{\delta_{N}}^{w}}^{*} E_{\varepsilon}\left|\frac{1}{\sqrt{N}}\sum_{i=1}^{N}\frac{\varepsilon_{i}}{w_{\theta}(X_{i})}\right|
\end{split}
\end{equation}
where
\[Z_{i}:=\frac{\delta_{Y_{i}}}{w_{\theta}(X_{i})}-P_{y,x}(f/w_{\theta}),\quad i=1,2,\dots, N.\]
Now, note that 
\[E_{\varepsilon} \left\lVert\frac{1}{\sqrt{N}}\sum_{i=1}^{N}\varepsilon_{i}Z_{i}\right\rVert_{\mathcal{F}_{\delta_{N}}^{w}}^{*}\overset{P}{\rightarrow}0\]
has already been shown in the final part of the proof of Lemma \ref{lemma_AEC_expectation_POISSON_B}, and that
\[\left\lVert P_{y,x}(f/w_{\theta})\right\rVert_{\mathcal{F}_{\delta_{N}}^{w}}^{*} E_{\varepsilon}\left|\frac{1}{\sqrt{N}}\sum_{i=1}^{N}\varepsilon_{i}\right|\leq\delta_{N}.\]
Thus, it remains to show that
\[\left\lVert \mathbb{P}_{y,N}\right\rVert_{\mathcal{F}_{\delta_{N}}^{w}}^{*} E_{\varepsilon}\left|\frac{1}{\sqrt{N}}\sum_{i=1}^{N}\frac{\varepsilon_{i}}{w_{\theta}(X_{i})}\right|\overset{P}{\rightarrow}0\]
as well. This will be done by showing that the left side goes to zero almost surely. To this aim note first that by Jensen's inequality and the independence of the Rademacher random variables
\[E_{\varepsilon}\left|\frac{1}{\sqrt{N}}\sum_{i=1}^{N}\frac{\varepsilon_{i}}{w_{\theta}(X_{i})}\right|\leq\sqrt{\frac{1}{N}\sum_{i=1}^{N}\frac{1}{w_{\theta}^{2}(X_{i})}},\]
and that the right side goes to a constant almost surely by assumptions C1 and B2$^{*}$. Then observe that
\[\left\lVert \mathbb{P}_{y,N}\right\rVert_{\mathcal{F}_{\delta_{N}}^{w}}^{*}\leq \left\lVert \mathbb{P}_{y,N}\right\rVert_{\mathcal{F}_{\delta_{N}}}^{*}\]
because the semimetric $\rho_{w}$ is stronger than $\rho$, that
\[\left\lVert \mathbb{P}_{y,N}\right\rVert_{\mathcal{F}_{\delta_{N}}}^{*}\leq \left\lVert \mathbb{P}_{y,N}-P_{y}\right\rVert_{\mathcal{F}_{\delta_{N}}}^{*}+\left\lVert P_{y}\right\rVert_{\mathcal{F}_{\delta_{N}}}^{*}\leq \left\lVert\mathbb{P}_{y,N}-P_{y}\right\rVert_{\mathcal{F}_{\delta_{N}}}^{*}+\delta_{N},\]
and finally that
\[\left\lVert\mathbb{P}_{y,N}-P_{y}\right\rVert_{\mathcal{F}_{\delta_{N}}}^{*}\overset{as}{\rightarrow}0\]
because $\mathcal{F}$ is a $P_{y}$-Donsker class (see Remark \ref{F_Donsker_implication}) and hence an outer almost sure $P_{y}$-Glivenko-Cantelli class.
\end{proof}

\begin{lem}[Almost sure version of conditional AEC]\label{lemma_AEC_outer_almost_sure_POISSON_B_HEP}
Let $\{\mathbf{S}_{N}\}_{N=1}^{\infty}$, $\mathcal{F}$ and $\{\mathbb{G}'_{N}\}_{N=1}^{\infty}$ be defined as in Lemma \ref{lemma_marginal_convergence_B_HEP}. Assume that conditions C1, B0 (the almost sure version), B1, B2$^{*}$, B3 and S' hold. Then it follows that
\[E_{d}\lVert \widetilde{\mathbb{G}}_{N}''\rVert_{\mathcal{F}_{\delta_{N}}^{w}}^{**}\overset{as}{\rightarrow} 0\quad\text{ for every }\delta_{N}\downarrow 0.\]
\end{lem}

\begin{proof}
Follow the steps of the proof of Lemma \ref{lemma_AEC_expectation_POISSON_B_HEP} up to display (\ref{risultato_intermedio_HEP}) (under the assumptions of the present lemma all occurrences of convergence in probability can be replaced by almost sure convergence) and note that in order to obtain the almost sure version of conditional AEC it is sufficient to show that
\begin{equation*}\label{risultato_intermedio_1_HEP}
E_{\varepsilon} \left\lVert\frac{1}{\sqrt{N}}\sum_{i=1}^{N}\varepsilon_{i}Z_{i}\right\rVert_{\mathcal{F}_{\delta_{N}}^{w}}^{*}:=E_{\varepsilon} \left\lVert\frac{1}{\sqrt{N}}\sum_{i=1}^{N}\varepsilon_{i}\left(\frac{\delta_{Y_{i}}}{w_{\theta}(X_{i})}-P_{y,x}(f/w_{\theta})\right)\right\rVert_{\mathcal{F}_{\delta_{N}}^{w}}^{*}
\end{equation*}
goes to zero almost surely (in fact, in the proof of Lemma \ref{lemma_AEC_expectation_POISSON_B_HEP} it has already been shown that the two remaining terms in (\ref{BBBound_HEP}) go to zero almost surely). This can be done by the method already seen in the proof of Lemma \ref{lemma_almost_sure_AEC_POISSON}. The details are left to the reader.
\end{proof}

Finally, it remains to deal with total boundedness. But this problem has already been handled in Subsection \ref{subsection_case_BBB} where it has been pointed out that assumptions B2$^{*}$ and B3 imply that $\mathcal{F}/w_{\theta}$ (and hence also $\mathcal{F}$) is totally bounded w.r.t. the semimetric $\rho_{w}$.

Now, having established sufficient conditions for convergence of the marginal distributions and for (conditional) AEC and total boundedness w.r.t. $\rho_{w}$, one can apply the methods of proof that have already been applied for Theorem \ref{unconditional_convergence}, Theorem \ref{outer_expectation_conditional_weak_convergence}, Corollary \ref{corollary_joint_weak_convergence} and Theorem \ref{outer_almost_sure_conditional_weak_convergence} in order to obtain the desired weak convergence results. As before, the weak convergence results will be stated for the HEP sequence rather than for $\{\widetilde{\mathbb{G}}_{N}''\}_{N=1}^{\infty}$.

\begin{thm}[Unconditional weak convergence]\label{HAJEK_unconditional_convergence_BBB}
Let $\{\mathbf{S}_{N}\}_{N=1}^{\infty}$ be the sequence of vectors of sample inclusion indicators corresponding to a sequence of Poisson sampling designs with first order sample inclusion probabilities defined as in (\ref{first_order_Poisson_special_case}). Let $\mathcal{F}$ be a class of measurable functions $f:\mathcal{Y}\mapsto\mathbb{R}$ and let $\{\mathbb{G}''_{N}\}_{N=1}^{\infty}$ be the sequence of HEPs corresponding to $\mathcal{F}$ and $\{\mathbf{S}_{N}\}_{N=1}^{\infty}$. Assume that conditions C1, B0 (the probability version suffices), B1 and B2$^{*}$ are satisfied. Then it follows that
\begin{itemize}
\item[(i) ] there exists zero-mean Gaussian process $\{\mathbb{G}''f:f\in\mathcal{F}\}$ with covariance function given by $\Sigma''$ which is a Borel measurable and tight random element of $l^{\infty}(\mathcal{F})$ such that
\[\mathbb{G}_{N}''\rightsquigarrow\mathbb{G}''\quad\text{ in }l^{\infty}(\mathcal{F});\]
\item[(ii) ] the sample paths $f\mapsto\mathbb{G}''f$ are uniformly $\rho_{w}$-continuous with probability $1$.
\end{itemize}
\end{thm}

\begin{thm}[Outer probability conditional weak convergence]\label{HAJEK_outer_expectation_conditional_weak_convergence_BBB}
Under the assumptions of Theorem \ref{HAJEK_unconditional_convergence_BBB} it follows that
\[\sup_{h\in BL_{1}(l^{\infty}(\mathcal{F}))}\left|E_{d}h(\mathbb{G}_{N}'')-Eh(\mathbb{G}'')\right|\overset{P*}{\rightarrow}0,\]
where $\mathbb{G}''$ is defined as in Theorem \ref{HAJEK_unconditional_convergence_BBB}.
\end{thm}

\begin{cor}[Joint weak convergence]\label{HAJEK_corollary_joint_weak_convergence}
Under the assumptions of Theorem \ref{HAJEK_unconditional_convergence_BBB} it follows that
\[(\mathbb{G}_{N}, \mathbb{G}_{N}'')\rightsquigarrow(\mathbb{G}, \mathbb{G}'')\text{ in }l^{\infty}(\mathcal{F})\times l^{\infty}(\mathcal{F}),\]
where $\mathbb{G}_{N}''$ and $\mathbb{G}''$ are defined as in Theorem \ref{HAJEK_unconditional_convergence_BBB}, $\mathbb{G}_{N}$ is the classical empirical process defined in (\ref{classical_empirical_process}), and where $\mathbb{G}$ is a Borel measurable and tight  $P_{y}$-Brownian Bridge which is independent from $\mathbb{G}''$.
\end{cor}

\begin{thm}[Outer almost sure conditional weak convergence]\label{HAJEK_outer_almost_sure_conditional_weak_convergence}
Let $\{\mathbf{S}_{N}\}_{N=1}^{\infty}$, $\mathcal{F}$ and $\{\mathbb{G}''_{N}\}_{N=1}^{\infty}$ be defined as in Theorem \ref{HAJEK_unconditional_convergence_BBB}. Assume that conditions C1, B0 (the almost sure version), B1, B2$^{*}$ and condition S' are satisfied. Then it follows that
\[\sup_{h\in BL_{1}(l^{\infty}(\mathcal{F}))}\left|E_{d}h(\mathbb{G}_{N}'')-Eh(\mathbb{G}'')\right|\overset{as*}{\rightarrow}0.\]
where $\mathbb{G}''$ is defined as in Theorem \ref{HAJEK_unconditional_convergence_BBB}.
\end{thm}

\section{Simulation results}\label{Simulation_results}

This section presents some simulation results for the above theory. The numerical results given in this section have been obtain by using the R Statistical Software \cite{citazioneR} in order to repeat $B=1000$ times the following steps:
\begin{itemize}
\item[1) ] Generate a population of $N$ independent observations $(Y_{i},X_{i})$ from the linear model $Y_{i}=X_{i}+U_{i}$, where the $X_{i}$'s are i.i.d. lognormal with $E(\ln X_{i})=0$ and $Var(\ln X_{i})=1$, and where the $U_{i}$'s are independent zero mean Gaussian random variables with $Var(U_{i})=X_{i}^{2}$, $i=1,2,\dots, N$.
\item[2) ] Select a sample $\mathbf{s}_{N}:=(s_{1,N}, s_{2,N}, \dots, s_{N,N})$ according to the Poisson sampling design with expected sample size $n_{N}$ and with first order sample inclusion probabilities $\pi_{i,N}$ proportional to the $X_{i}$ values (this step was performed by using the function "UPpoisson" from the R package "sampling" \citep{citazione_sampling_package}).
\item[3) ] Compute the Horvitz-Thompson and the H{\'a}jek estimator for the population cdf $F_{Y,N}(t):=\sum_{i=1}^{N}I(Y_{i}\leq t)/N$, $t\in\mathbb{R}$, and compute the uniform distance between each of those estimators and $F_{Y,N}$, i.e. compute $\lVert \mathbb{G}'_{N}\rVert_{\mathcal{F}}$ and $\lVert \mathbb{G}''_{N}\rVert_{\mathcal{F}}$ for the case where $\mathcal{F}:=\{I(y\leq t):t\in\mathbb{R}\}$.
\item[4) ] Estimate the $\gamma$-quantiles $q_{\gamma}'$ and $q_{\gamma}''$ of the limiting distributions of $\lVert \mathbb{G}'_{N}\rVert_{\mathcal{F}}$ and $\lVert \mathbb{G}''_{N}\rVert_{\mathcal{F}}$, i.e. the $\gamma$-quantiles of the distributions of $\lVert \mathbb{G}'\rVert_{\mathcal{F}}$ and $\lVert \mathbb{G}''\rVert_{\mathcal{F}}$. The procedure for doing this is based on Algorithm 5.1 in \citep{Kroese} which was also used in the simulation study in \cite{Bertail_2017}. The details are described below.
\item[5) ] Compute the asymptotic uniform $\gamma$-confidence bands for the population cdf $F_{Y,N}$ based on the Horvitz-Thompson and the H{\'a}jek estimators and verify whether $F_{Y,N}$ lies within these confidence bands, i.e. verify whether $\lVert \mathbb{G}'_{N}\rVert_{\mathcal{F}}\leq \widehat{q}_{\gamma}'$ and whether $\lVert \mathbb{G}''_{N}\rVert_{\mathcal{F}}\leq \widehat{q}_{\gamma}''$, where $\widehat{q}_{\gamma}'$ and $\widehat{q}_{\gamma}''$ are the estimates of $q_{\gamma}'$ and $q_{\gamma}''$, respectively, which have already been computed at step 4. Note that the widths of the two asymptotic uniform $2\gamma$-confidence bands are given by $2\widehat{q}_{\gamma}'$ and $2\widehat{q}_{\gamma}''$, respectively.  
\end{itemize}

The $\gamma$-quantiles of the distributions of $\lVert \mathbb{G}'\rVert_{\mathcal{F}}$ and $\lVert \mathbb{G}''\rVert_{\mathcal{F}}$ were estimated according to the following procedure:
\begin{itemize}
\item[i) ] Estimate the covariance matrices $\Sigma'(\mathbf{f})$ and $\Sigma''(\mathbf{f})$ for $\mathbf{f}:=(I(y\leq Y_{i_{1}}), I(y\leq Y_{i_{2}}), \dots, I(y\leq Y_{i_{r}}))^{\intercal}$ where $(i_{1}, i_{2}, \dots, i_{r})$ correspond to the sampled population units, i.e. $(i_{1}, i_{2}, \dots, i_{r})$ are the values of the subscript $i$ for which $s_{i,N}=1$, $i=1,2,\dots, N$. The components $\Sigma'_{i,j}(\mathbf{f})$ and $\Sigma''_{i,j}(\mathbf{f})$, $i,j=i_{1}, i_{2}, \dots, i_{r}$, of the two covariance matrices were estimated as follows:
\[\widehat{\Sigma}'_{i,j}(\mathbf{f}):=\frac{1}{N}\sum_{k=1}^{N}s_{k,N}\frac{1-\pi_{k,N}}{\pi_{k,N}^{2}}I(Y_{k}\leq Y_{i})I(Y_{k}\leq Y_{j})\]
and
\[\widehat{\Sigma}''_{i,j}(\mathbf{f}):=\frac{1}{\sum_{k=1}^{N}\frac{s_{k,N}}{\pi_{k,N}}}\sum_{k=1}^{N}s_{k,N}\frac{1-\pi_{k,N}}{\pi_{k,N}^{2}}[I(Y_{k}\leq Y_{i})-\overline{I}_{i}][I(Y_{k}\leq Y_{j})-\overline{I}_{j}]\]
with
\[\overline{I}_{i}:=\frac{1}{\sum_{k=1}^{N}\frac{s_{k,N}}{\pi_{k,N}}}\sum_{k=1}^{N}\frac{s_{k,N}}{\pi_{k,N}}I(Y_{k}\leq Y_{i}).\]
\item[ii) ] Compute the Cholesky decompositions of the estimated covariance matrices, i.e. compute two lower triangular matrices $L$ and $H$ such that $\widehat{\Sigma}'(\mathbf{f})=LL^{\intercal}$ and $\widehat{\Sigma}''_{i,j}(\mathbf{f})=HH^{\intercal}$.
\item[iii) ] Generate independently $1000$ random vectors $\mathbf{Z}_{b}:=(Z_{1,b}, Z_{2,b}, \dots, Z_{r,b})^{\intercal}$, $b=1,2,\dots, 1000$, whose components $Z_{k,b}$ are i.i.d. standard normal random variables and compute the vectors $\mathbf{G}_{b}':=L\mathbf{Z}_{b}$ and $\mathbf{G}_{b}'':=H\mathbf{Z}_{b}$ which can be considered as realizations of the limit processes $\mathbb{G}'$ and $\mathbb{G}''$, respectively.
\item[iv) ] for each $b=1,2,\dots, 1000$ compute the maximum norms $\lVert \mathbf{G}_{b}'\rVert_{\infty}$ and $\lVert \mathbf{G}_{b}''\rVert_{\infty}$ (i.e. the two maxima of the absolute values of the components of $\mathbf{G}_{b}'$ and $\mathbf{G}_{b}''$), put the two vectors $(\lVert \mathbf{G}_{1}'\rVert_{\infty}, \lVert \mathbf{G}_{2}'\rVert_{\infty}, \dots, \lVert \mathbf{G}_{1000}'\rVert_{\infty})$ and $(\lVert \mathbf{G}_{1}''\rVert_{\infty}, \lVert \mathbf{G}_{2}''\rVert_{\infty}, \dots, \lVert \mathbf{G}_{1000}''\rVert_{\infty})$ in ascending order and set $\widehat{q}_{\gamma}'$ equal to the $\gamma$-quantile of the first vector, and set $\widehat{q}_{\gamma}''$ equal to the $\gamma$-quantile of the second vector.
\end{itemize}

\begin{table}[h]\label{tabella_simulazioni}
\caption{Simulation results for the Horvitz-Thompson empirical process.}
\label{tabella_simulazioni_HTEP}
\begin{tabular}{rccc}
\hline
 & $\gamma=0.90$ & $\gamma=0.95$ & $\gamma=0.99$\\
\hline
$\mathbf{N=1000}$ \\
\multirow{ 2}{*}{$\alpha=0.05$} & 0.849& 0.901& 0.948\\
&(0.9123; 16.5467)&(1.0573; 19.6398)&(1.3429; 24.9830)\\
\multirow{ 2}{*}{$\alpha=0.10$} & 0.846& 0.912& 0.959\\
&(0.5853; 1.8251)&(0.6738; 2.1567)&(0.8506; 2.6515)\\
$\mathbf{N=2000}$ \\
\multirow{ 2}{*}{$\alpha=0.05$} & 0.860& 0.919& 0.957\\
&(0.5967; 2.2616)&(0.6883; 2.6380)&(0.8660; 3.2897)\\
\multirow{ 2}{*}{$\alpha=0.10$} & 0.865& 0.929& 0.978\\
&(0.4263; 1.1658)&(0.4899; 1.3830)&(0.6158; 1.6729)\\
$\mathbf{N=4000}$ \\
\multirow{ 2}{*}{$\alpha=0.05$} & 0.854& 0.916& 0.965\\
&(0.4296; 1.2044)&(0.4940; 1.3863)&(0.6201; 1.7086)\\
\multirow{ 2}{*}{$\alpha=0.10$} & 0.870 & 0.928& 0.976\\
&(0.3065; 0.5912)&(0.3521; 0.6815)&(0.4407; 0.8974)\\
\hline
\end{tabular}
\end{table}

\begin{table}[h]\label{tabella_simulazioni}
\caption{Simulation results for the H{\'a}jek empirical process.}
\label{tabella_simulazioni_HEP}
\begin{tabular}{rccc}
\hline
 & $\gamma=0.90$ & $\gamma=0.95$ & $\gamma=0.99$\\
\hline
$\mathbf{N=1000}$ \\
\multirow{ 2}{*}{$\alpha=0.05$} & 0.744& 0.833& 0.935\\
&(0.4579; 1.1284)&(0.5195; 1.3473)&(0.6403; 1.6996)\\
\multirow{ 2}{*}{$\alpha=0.10$} & 0.804& 0.878& 0.940\\
&(0.3477; 0.7969)&(0.3927; 0.9365)&(0.4813; 1.2784)\\
$\mathbf{N=2000}$ \\
\multirow{ 2}{*}{$\alpha=0.05$} & 0.792& 0.866& 0.944\\
&(0.3526; 0.7526)&(0.3984; 0.8804)&(0.4890; 1.1742)\\
\multirow{ 2}{*}{$\alpha=0.10$} & 0.844& 0.913& 0.967\\
&(0.2622; 0.5141)&(0.2953; 0.6119)&(0.3611; 0.7657)\\
$\mathbf{N=4000}$ \\
\multirow{ 2}{*}{$\alpha=0.05$} & 0.815& 0.888& 0.958\\
&(0.2632; 0.6370)&(0.2964; 0.7045)&(0.3619; 0.8760)\\
\multirow{ 2}{*}{$\alpha=0.10$} & 0.847& 0.914& 0.967\\
&(0.1928; 0.3740)&(0.2164; 0.4237)&(0.2631; 0.5214)\\
 
\hline
\end{tabular}
\end{table}

Table \ref{tabella_simulazioni_HTEP} (for the HTEP) and Table \ref{tabella_simulazioni_HEP} (for the HEP) summarize the simulation results. For each considered population size $N=1000, 2000, 4000$, for each considered sampling fraction $\alpha:=n_{N}/N=0.05, 0.10$ and for each considered confidence level $\gamma=0.90, 0.95, 0.99$, the two tables report the estimate of the coverage probability of the corresponding confidence band for $F_{Y,N}$ as well as the average (the first figure within each bracket) and the maximum width (the second figure within each bracket) of the $B=1000$ simulated confidence bands. The simulation results show that the confidence bands based on the HTEP are often much too wide to be useful. The confidence bands for the HEP are much narrower but nevertheless their width is occasionally larger than $1$. Other simulation results not reported here show that with more evenly distributed first order inclusion probabilities the confidence bands would have been much narrower. As for the estimated coverage probabilities, they are always smaller than the nominal confidence level $\gamma$, and those obtained from the HEP are only a little bit smaller than those obtained from the HTEP even though the widths of the confidence bands obtained from the latter appear to be much larger. However, other simulation results (not reported here) suggest that the estimated coverage probabilities get much closer to the nominal confidence levels as the variability in the first order sample inclusion probabilities decreases.

%
%
%
%
%



\bibliography{arXiv_WEAK_CONVERGENCE_POISSON_SAMPLING_version3}

\begin{thebibliography}{13}

\bibitem[\protect\citeauthoryear{Bertail, Chautru and
  Cl\'{e}men\c{c}on}{2017}]{Bertail_2017}
\begin{barticle}[author]
\bauthor{\bsnm{Bertail},~\bfnm{Patrice}\binits{P.}},
  \bauthor{\bsnm{Chautru},~\bfnm{Emilie}\binits{E.}} \AND
  \bauthor{\bsnm{Cl\'{e}men\c{c}on},~\bfnm{Stephan}\binits{S.}}
(\byear{2017}).
\btitle{Empirical processes in survey sampling with (conditional) {P}oisson
  designs}.
\bjournal{Scand. J. Stat.}
\bvolume{44}
\bpages{97--111}.
\bdoi{10.1111/sjos.12243}
\bmrnumber{3619696}
\end{barticle}
\endbibitem

\bibitem[\protect\citeauthoryear{Boistard, Lopuha\"{a} and
  Ruiz-Gazen}{2017}]{Boistard_2017}
\begin{barticle}[author]
\bauthor{\bsnm{Boistard},~\bfnm{H\'{e}l\`ene}\binits{H.}},
  \bauthor{\bsnm{Lopuha\"{a}},~\bfnm{Hendrik~P.}\binits{H.~P.}} \AND
  \bauthor{\bsnm{Ruiz-Gazen},~\bfnm{Anne}\binits{A.}}
(\byear{2017}).
\btitle{Functional central limit theorems for single-stage sampling designs}.
\bjournal{Ann. Statist.}
\bvolume{45}
\bpages{1728--1758}.
\bdoi{10.1214/16-AOS1507}
\bmrnumber{3670194}
\end{barticle}
\endbibitem

\bibitem[\protect\citeauthoryear{Breslow and
  Wellner}{2007}]{Breslow_Wellner_2007}
\begin{barticle}[author]
\bauthor{\bsnm{Breslow},~\bfnm{Norman~E.}\binits{N.~E.}} \AND
  \bauthor{\bsnm{Wellner},~\bfnm{Jon~A.}\binits{J.~A.}}
(\byear{2007}).
\btitle{Weighted likelihood for semiparametric models and two-phase stratified
  samples, with application to {C}ox regression}.
\bjournal{Scand. J. Statist.}
\bvolume{34}
\bpages{86--102}.
\bdoi{10.1111/j.1467-9469.2006.00523.x}
\bmrnumber{2325244}
\end{barticle}
\endbibitem

\bibitem[\protect\citeauthoryear{Conti}{2014}]{Conti_2014}
\begin{barticle}[author]
\bauthor{\bsnm{Conti},~\bfnm{Pier~Luigi}\binits{P.~L.}}
(\byear{2014}).
\btitle{On the estimation of the distribution function of a finite population
  under high entropy sampling designs, with applications}.
\bjournal{Sankhya B}
\bvolume{76}
\bpages{234--259}.
\bdoi{10.1007/s13571-014-0083-x}
\bmrnumber{3302272}
\end{barticle}
\endbibitem

\bibitem[\protect\citeauthoryear{Kosorok}{2008}]{Kosorok}
\begin{bbook}[author]
\bauthor{\bsnm{Kosorok},~\bfnm{Michael~R.}\binits{M.~R.}}
(\byear{2008}).
\btitle{Introduction to empirical processes and semiparametric inference}.
\bseries{Springer Series in Statistics}.
\bpublisher{Springer, New York}.
\bdoi{10.1007/978-0-387-74978-5}
\bmrnumber{2724368}
\end{bbook}
\endbibitem

\bibitem[\protect\citeauthoryear{Kroese, Taimre and Botev}{2011}]{Kroese}
\begin{bbook}[author]
\bauthor{\bsnm{Kroese},~\bfnm{Dirk~P.}\binits{D.~P.}},
  \bauthor{\bsnm{Taimre},~\bfnm{Thomas}\binits{T.}} \AND
  \bauthor{\bsnm{Botev},~\bfnm{Zdravko~I.}\binits{Z.~I.}}
(\byear{2011}).
\btitle{Handbook of Monte Carlo Methods}.
\bseries{Wiley Series in Probability and Statistics}.
\bpublisher{John Wiley \& Sons, New York}.
\end{bbook}
\endbibitem

\bibitem[\protect\citeauthoryear{Pollard}{1990}]{Pollard}
\begin{bbook}[author]
\bauthor{\bsnm{Pollard},~\bfnm{David}\binits{D.}}
(\byear{1990}).
\btitle{Empirical processes: theory and applications}.
\bseries{NSF-CBMS Regional Conference Series in Probability and Statistics}
\bvolume{2}.
\bpublisher{Institute of Mathematical Statistics, Hayward, CA; American
  Statistical Association, Alexandria, VA}.
\bmrnumber{1089429}
\end{bbook}
\endbibitem

\bibitem[\protect\citeauthoryear{Pr{\ae}stgaard and
  Wellner}{1993}]{Praestgaard_Wellner_1993}
\begin{barticle}[author]
\bauthor{\bsnm{Pr{\ae}stgaard},~\bfnm{Jens}\binits{J.}} \AND
  \bauthor{\bsnm{Wellner},~\bfnm{Jon~A.}\binits{J.~A.}}
(\byear{1993}).
\btitle{Exchangeably weighted bootstraps of the general empirical process}.
\bjournal{Ann. Probab.}
\bvolume{21}
\bpages{2053--2086}.
\bmrnumber{1245301}
\end{barticle}
\endbibitem

\bibitem[\protect\citeauthoryear{Saegusa and
  Wellner}{2013}]{Saegusa_Wellner_2013}
\begin{barticle}[author]
\bauthor{\bsnm{Saegusa},~\bfnm{Takumi}\binits{T.}} \AND
  \bauthor{\bsnm{Wellner},~\bfnm{Jon~A.}\binits{J.~A.}}
(\byear{2013}).
\btitle{Weighted likelihood estimation under two-phase sampling}.
\bjournal{Ann. Statist.}
\bvolume{41}
\bpages{269--295}.
\bdoi{10.1214/12-AOS1073}
\bmrnumber{3059418}
\end{barticle}
\endbibitem

\bibitem[\protect\citeauthoryear{{R Core Team}}{2013}]{citazioneR}
\begin{bmanual}[author]
\bauthor{\bsnm{{R Core Team}}}
(\byear{2013}).
\btitle{R: A Language and Environment for Statistical Computing}
\bpublisher{R Foundation for Statistical Computing},
\baddress{Vienna, Austria}.
\end{bmanual}
\endbibitem

\bibitem[\protect\citeauthoryear{Till{\'e} and
  Matei}{2016}]{citazione_sampling_package}
\begin{bmanual}[author]
\bauthor{\bsnm{Till{\'e}},~\bfnm{Yves}\binits{Y.}} \AND
  \bauthor{\bsnm{Matei},~\bfnm{Alina}\binits{A.}}
(\byear{2016}).
\btitle{sampling: Survey Sampling}
\bnote{R package version 2.8}.
\end{bmanual}
\endbibitem

\bibitem[\protect\citeauthoryear{van~der Vaart}{1998}]{vdV}
\begin{bbook}[author]
\bauthor{\bparticle{van~der} \bsnm{Vaart},~\bfnm{A.~W.}\binits{A.~W.}}
(\byear{1998}).
\btitle{Asymptotic statistics}.
\bseries{Cambridge Series in Statistical and Probabilistic Mathematics}
\bvolume{3}.
\bpublisher{Cambridge University Press, Cambridge}.
\bdoi{10.1017/CBO9780511802256}
\bmrnumber{1652247}
\end{bbook}
\endbibitem

\bibitem[\protect\citeauthoryear{van~der Vaart and Wellner}{1996}]{vdVW}
\begin{bbook}[author]
\bauthor{\bparticle{van~der} \bsnm{Vaart},~\bfnm{Aad~W.}\binits{A.~W.}} \AND
  \bauthor{\bsnm{Wellner},~\bfnm{Jon~A.}\binits{J.~A.}}
(\byear{1996}).
\btitle{Weak convergence and empirical processes}.
\bseries{Springer Series in Statistics}.
\bpublisher{Springer-Verlag, New York}
\bnote{With applications to statistics}.
\bdoi{10.1007/978-1-4757-2545-2}
\bmrnumber{1385671}
\end{bbook}
\endbibitem

\end{thebibliography}

\end{document}